\documentclass[12pt]{article}

\usepackage{a4wide}
\usepackage{amsmath,amsfonts,amssymb}
\usepackage{bm}
\usepackage{color}
\usepackage[english]{babel}
\usepackage{graphicx}
\usepackage{todonotes}

\numberwithin{equation}{section}

\newcommand{\bfI}{\boldsymbol{I}}

\newcommand{\tn}{|||}

\newcommand{\mcK}{\mathcal{K}}

\newcommand{\mcN}{\mathcal{N}}
\newcommand{\mcF}{\mathcal{F}}

\newcommand{\IR}{\mathbb{R}}

\newtheorem{thm}{Theorem}[section]
\newtheorem{rem}{Remark}[section]

\newenvironment{proof}{\noindent \newline {\bf Proof.}}
{\hfill \mbox{\fbox{} } \newline}

\title{\bf A Cut Finite Element Method with Boundary 
Value Correction}
\date{\today}
\author{Erik Burman,\mbox{ } Peter Hansbo,\mbox{ } Mats G. Larson}

\begin{document}

\maketitle

\begin{abstract} In this contribution we develop a cut finite element 
method with boundary value correction of the type originally proposed 
by Bramble, Dupont, and Thom\'ee in \cite{BrDuTh72}. The cut finite 
element method is a fictitious domain method with Nitsche type 
enforcement of Dirichlet conditions together with stabilization of 
the elements at the boundary which is stable and enjoy optimal order approximation properties. A computational difficulty is, however, the geometric computations related to quadrature on the cut elements which 
must be accurate enough to achieve higher order approximation. With boundary value correction we may use only a piecewise linear 
approximation of the boundary, which is very convenient in a cut finite element method, and still obtain optimal order convergence. The boundary 
value correction is a modified Nitsche formulation involving a Taylor expansion in the normal direction compensating for the approximation of the boundary. Key to the analysis is a consistent stabilization term 
which enables us to prove stability of the method and a priori 
error estimates with explicit dependence on the meshsize and distance 
between the exact and approximate boundary.
\end{abstract}

\section{Introduction}

We consider a cut finite element method (CutFEM) for a second order 
elliptic boundary value problem with Dirichlet conditions. In standard 
fictitious domain CutFEM the boundary is represented on a background 
grid and allowed to cut through the elements in an arbitrary fashion. 
The Dirichlet conditions are enforced weakly using Nitsche's method
\cite{Nit70}.  
We refer to \cite{BeBuHa09}, \cite{Bu10}, \cite{BuHa12}, 
\cite{MasLarLog14}, \cite{JoLa13}, for recent developments of this approach. See also the recent overview paper \cite{BurClaHan15} and 
\cite{MasLarLog13} for implementation issues.

Cut finite element methods is one way of alleviating the problem of mesh generation and
allowing for more structured meshes and associated solvers. For this reason, the interest for such methods has increased significantly
during the last few years; among recent contributions we mention the finite cell method of Parvizian, D\"uster, et al. \cite{PaDuRa07,DuPaYaRa08};
the least squares stabilized Lagrange multiplier methods of Haslinger
and Renard \cite{HaRe09}, Tur et al. \cite{TuAlNaRo14}, and Baiges et al. \cite{BaCoHeShWa12}; the stabilization of
Nitsche's method by Codina and Baiges \cite{CoBa09};
the local projection stabilization of multipliers of Barrenechea and Chouly \cite{BaCh12} and of Amdouni, Moakher, and Renard \cite{AmMoRe14}.

In this contribution we develop a version of CutFEM based on the 
idea of boundary value correction originally proposed for standard
finite element methods on an approximate domain in \cite{BrDuTh72} and
further developed in \cite{Du74}. Using the closest point mapping to
the exact boundary, or an approximation thereof, the boundary condition on the exact boundary may be weakly enforced using Nitsche's method on the boundary of the approximate 
domain. A Taylor expansion is used to approximate the value of the solution on the exact boundary in terms of the value and normal derivatives at the discrete approximate boundary. Key to the stability 
of the method is a consistent stabilization term that, also 
in the case of arbitrary cut elements at the boundary, provide control 
of the variation of the function in the vicinity of the boundary. More precisely, the stabilization ensures that the inverse inequality necessary to prove coercivity holds and that the resulting linear system of 
equations has the optimal condition number $O(h^{-2})$, where $h$ is the 
mesh parameter, independent of the position of the boundary on the background grid.

We prove optimal order a priori error estimates, in the energy and $L^2$ norms, in terms of the error in the boundary approximation and the meshsize. Of particular practical importance is the fact that  we 
may use a piecewise linear approximation of the boundary, which is 
very convenient from a computational point of view since the geometric computations are simple in this case and a piecewise linear distance function may be used to construct the 
discrete domain. We obtain optimal order convergence for higher order polynomial approximation of the solution if the Taylor expansion has sufficiently high order. In particular, for second and third order polynomials we obtain optimal order error estimates in the energy and 
$L^2$ norms with only one term in the Taylor expansion. Note that 
without boundary correction one typically requires $O(h^{p+1})$ accuracy 
in the $L^\infty$ norm for the approximation of the domain which leads 
to significantly more involved computations on the cut elements for 
higher order elements, see \cite{JoLa13}. However, also in the case of no boundary value correction our analysis in fact provides optimal order error estimates if the approximation of the boundary is accurate enough and thus we obtain 
an analysis for the standard cut finite element method with approximate boundary. Finally, we also prove estimates for the error both on the discrete domain and 
on the exact domain. The discrete solution on the exact domain is directly 
defined by the method since we may include all elements that intersect 
the union of the discrete and exact domains in the active mesh. Even 
though some active elements may not intersect the discrete domain the resulting method is stable due to the stabilization term and no 
auxiliary extension of the discrete solution outside of the discrete domain is necessary. We present numerical results illustrating our theoretical findings. 

The outline of the paper is as follows: In Section 2 
we formulate the model problem and our method, in Section 3 we present 
our theoretical analysis, and in Section 4 we present the numerical results.

\section{Model Problem and Method}

\subsection{The Domain}

Let $\Omega$ be a domain in $\mathbb{R}^d$ with smooth boundary 
$\partial \Omega$ and exterior unit normal $n$. We let $\rho$ 
be the signed distance function, negative on the inside and 
positive on the outside, to $\partial \Omega$ and we let 
$U_\delta(\partial \Omega)$ 
be the tubular neighborhood $\{x\in \IR^d : |\rho(x)| < \delta\}$ 
of $\partial \Omega$. Then there is a constant $\delta_0>0$ such 
that the closest point mapping $p(x):U_{\delta_0}(\partial \Omega) 
\rightarrow \partial \Omega$ is well defined and we have the 
identity $p(x) = x - \rho(x)n(p(x))$. We assume that $\delta_0$ is
chosen small enough that $p(x)$ is a bijection. See \cite{GilTru01}, 
Section 14.6 for further details on distance functions.

\subsection{The Model Problem}
We consider the problem: find $u:\Omega \rightarrow \IR$ 
such that
\begin{alignat}{2}\label{eq:poissoninterior}
-\Delta u &= f \qquad 
&& \text{in $\Omega$}
\\ \label{eq:poissonbc}
u &= g \qquad && \text{on $\partial\Omega$}
\end{alignat}
where $f\in H^{-1}(\Omega)$ and $g\in H^{1/2}(\partial \Omega)$ 
are given data. It follows from the Lax-Milgram Lemma that there 
exists a unique solution to this problem and we also have the 
elliptic regularity estimate
\begin{equation}\label{eq:ellipticregularity}
\|u\|_{H^{s+2}(\Omega)} \lesssim \|f\|_{H^s(\Omega)}, \qquad 
s \geq -1
\end{equation} 
Here and below we use the notation $\lesssim$ to denote less or 
equal up to a constant.

\subsection{The Mesh, Discrete Domains, and Finite Element Spaces}
\begin{itemize}
\item Let $\Omega_0$ be a convex polygonal domain such 
that $U_{\delta_0}(\Omega) \subset \IR^d$, where 
$U_{\delta}(\Omega) = U_{\delta}(\partial \Omega) \cup 
\Omega$. Let $\mcK_{0,h}, h \in (0,h_0]$, 
be a family of quasiuniform partitions, with 
mesh parameter $h$, of $\Omega_0$ into shape 
regular triangles or tetrahedra $K$. We refer 
to $\mcK_{h,0}$ as the background mesh.

\item Let $\Omega_h$, $h \in (0,h_0]$, be a 
family of polygonal domains approximating $\Omega$. To each $
\Omega_h$ we associate the functions $\nu_h:\partial \Omega_h \rightarrow
\mathbb{R}^d$, $|\nu_h|=1$, and $\varrho_h:\partial \Omega_h \rightarrow \mathbb{R}$, such that if $p_h(x,\varsigma):=x + \varsigma \nu_h(x)$ then
$p_h(x,\varrho_h(x)) \in \partial \Omega$ for all $x \in \partial \Omega_h$. We will also assume that $p_h(x,\varsigma)
\in U_{\delta_0}(\Omega)$ for all $x \in \partial \Omega_h$ and all
$\varsigma$ between $0$ and $\varrho_h(x)$. For conciseness we will 
drop the second argument of $p_h$ below whenever it takes the value $\varrho_h(x)$. We assume that the following assumptions are satisfied

\begin{equation}\label{eq:geomassum-a}
\delta_h := \| \varrho_h \|_{L^\infty(\partial \Omega_h)} = o(h), 
\qquad  h \in (0,h_0]
\end{equation}
and 
\begin{equation}\label{eq:geomassum-c}
\| \nu_h - n\circ p \|_{L^\infty(\partial \Omega_h)} = o(1), 
\qquad  h \in (0,h_0]
\end{equation}
where $o(\cdot)$ denotes the little ordo. We also assume 
that $h_0$ is small enough to guarantee that 
\begin{equation}\label{eq:geomassum-b}
\partial \Omega_h 
\subset U_{\delta_0}(\partial \Omega), \qquad h\in(0,h_0]
\end{equation}
and that  there exists $M>0$ such for any $y \in U_{\delta_0}(\partial
\Omega)$ the equation, find $x \in \partial \Omega_h$ and $
|\varsigma| \leq \delta_h$ such that
\begin{equation}\label{eq:assump_olap}
p_h(x,\varsigma) = y
\end{equation}
has a solution set $\mathcal{P}_h$ with
\begin{equation}\label{eq:card_hyp}
\mbox{card}(\mathcal{P}_h) \leq M
\end{equation}
uniformly in $h$. The rationale of this assumption is to ensure that
even if $p_h$ is not a bijection its image can not degenerate for
vanishing $h$.

\paragraph{Choice of $\boldsymbol \nu_h$.}
During computation, typically the quantities that are
easily accessible on $\partial \Omega_h$ are $n_h$ and $\rho$. 
The two choices that are natural for $\nu_h, \varrho_h$ are therefore
$\nu_h:=n_h$, $\varrho_h:=\varsigma$, with $\varsigma$ solution to
$\rho(p_h(x,\varsigma))=0$
or $\nu_h:= n \circ p$ and $\varrho_h:= \rho$. Both cases requires the
solution of nonlinear equations. The computation of $\varrho_h$ using
Newton's method in the first case is substantially less costly than
that of $n \circ p$, since the first quantity is a scalar and the
initial guess $\rho$ is more accurate.

Observe that if $\nu_h := n \circ p$ then the mapping $p_h$ coincides
with $p(x)$. It is therefore a bijection and all the above assumptions hold by
the properties of the closest point mapping. This bijection property does not hold in the general case. However, we assume that the equation $\rho(p_h(x,\varsigma)) = 0$ has at least one
solution for every $x \in \partial \Omega_h$ and $\varrho_h$ may then
be identified with the solution of smallest magnitude.
As an example consider the practically important 
case where $\partial \Omega_h$ is defined by the zero level set of a piecewise linear nodal interpolant of the distance function and we 
choose $\nu_h := n_h$,
with $n_h$ denoting the normal of $\partial \Omega_h$. That the
associated 
$\varrho_h$ exists for all $x \in \partial \Omega_h$ follows
immediately from the implicit function theorem: the equation in
$\varsigma$, $\rho(x + \varsigma
n \circ p)=0$ has a solution since $p$ is a bijection and then so does
$\rho(x + 
\varsigma n_h)=0$ since $\nabla \rho \cdot n_h > 0$ for $h$ small
enough. The assumption (\ref{eq:card_hyp})
must clearly hold in this case, since if it does not then also $p$ must have a
critical point in $U_{\delta_0}(\partial
\Omega)$ (since $p_h \rightarrow p$ and the number of solutions
is bounded below as $h \rightarrow 0$), but this contradicts the fact 
that $p$ is a bijection. Moreover we have the estimates
\begin{equation}
\delta_h \lesssim h^2, \qquad \|\nu_h - n \circ p \|_{L^\infty(\partial \Omega_h)} \lesssim h 
\end{equation}

\item Given a subset $\omega$ of $\Omega_0$, let 
$\mcK_h(\omega)$ be the submesh defined by
\begin{equation}
\mcK_{h}(\omega) = \{K \in \mcK_{0,h} : \overline{K} \cap 
\overline{\omega} 
\neq \emptyset \}
\end{equation}
i.e., the submesh consisting of elements that intersect 
$\overline{\omega}$, and let 
\begin{equation}
\mcN_{h}(\omega) = \cup_{K \in \mcK_{h}(\omega)} K
\end{equation}
be the union of all elements in $\mcK_h(\omega)$. Below the $L^2$-norm 
of discrete functions frequently should be interpreted as the broken norm. For example for norms over $\mcN_{h}$ we have
\begin{equation}
\|v\|_{\mcN_{h}(\omega)}^2 := \sum_{K \in \mcK_{h}(\omega)} \|v\|_K^2
\end{equation}

\item Let the active mesh $\mcK_h$ be defined by
\begin{equation}
\mcK_{h} := \mcK_h(\Omega \cup \Omega_h)
\end{equation}
i.e., the submesh consisting of elements that intersect 
$\Omega_h\cup\Omega$, and let 
\begin{equation}
\mcN_{h} := \mcN_h(\Omega\cup \Omega_h)
\end{equation}
be the union of all elements in $\mcK_h$.
 \item Let $V_{0,h}$ be the space of piecewise continuous 
polynomials of order $p$ defined on $\mcK_{0,h}$ and let the 
finite element space be defined by
\begin{equation}
V_{h} := V_{0,h}|_{\mcN_{h}}
\end{equation}
\end{itemize}

\subsection{Extensions}
\label{sec:extension}
There is an extension operator 
$E:H^s(\Omega) \rightarrow H^s(U_{\delta_0}(\Omega))$ 
such that 
\begin{equation}\label{eq:extensionstability}
\| Ev \|_{H^s(U_\delta(\Omega))} \lesssim \| v \|_{H^s(\Omega)}, 
\qquad s\geq 0 
\end{equation}
see \cite{Fo95}. For brevity we shall use the notation 
$v$ for the extended function as well, i.e., $v=Ev$ on 
$U_{\delta_0}(\Omega)$.

\subsection{The Method} 

\paragraph{Derivation.} Let $f=Ef$ and $u=Eu$ be the 
extensions of $f$ and $u$ from $\Omega$ to 
$U_{\delta_0}(\Omega)$. For $v \in V_h$ we 
have using Green's formula
\begin{align}
(f,v)_{\Omega_h} 
&= (f+\Delta u,v)_{\Omega_h} - (\Delta u,v)_{\Omega_h} 
\\
&=(f +\Delta u ,v)_{\Omega_h \setminus \Omega} 
+ (\nabla u,\nabla v)_{\Omega_h} 
- (n_h\cdot\nabla u,v)_{\partial \Omega_h}
\end{align}
where we used the fact $f+\Delta u=0$ on $\Omega$, while 
on $\Omega_h \setminus \Omega$ we have 
$f+ \Delta u = Ef - \Delta Eu$, which is not in general equal 
to zero. Now the boundary condition $u=g$ on $\partial \Omega$ 
may be enforced weakly as follows
\begin{align}
(f,v)_{\Omega_h}&=(f +\Delta u ,v)_{\Omega_h} 
+ (\nabla u,\nabla v)_{\Omega_h} 
- (n_h\cdot\nabla u,v)_{\partial \Omega_h}
\\ \nonumber
&\qquad - (u\circ p_h - g\circ p_h,n_h\cdot \nabla v)_{\partial \Omega_h}
+ \beta h^{-1} (u\circ p_h - g\circ p_h, v)_{\partial \Omega_h}  
\end{align}
The positive constant $\beta$ must be chosen large enough
to ensure stability, cf. below.

Since we do not have access to $u\circ p_h$ we use a Taylor 
approximation in the direction $\nu_h$
\begin{equation}\label{def:Taylor}
u\circ p_h(x) \approx T_k(u)(x) 
:= \sum_{j=0}^k \frac{D_{\nu_h}^j v(x)}{j!}\varrho_h^j(x)
\end{equation}
where $D_{\nu_h}^j$ is the $j$:th partial derivative in 
the direction $\nu_h$.
Thus it follows that the solution to 
(\ref{eq:poissoninterior})-(\ref{eq:poissonbc}) satisfies 
\begin{align}\label{eq:derivation-c}
(f,v)_{\Omega_h}&=(f +\Delta u ,v)_{\Omega_h} + (\nabla u,\nabla v)_{\Omega_h} 
- (n_h\cdot\nabla u,v)_{\partial \Omega_h}
\\ \nonumber
&\qquad \qquad - (T_k(u) - g\circ p_h,n_h\cdot \nabla v)_{\partial \Omega_h}
+ \beta h^{-1} (T_k(u) - g\circ p_h, v)_{\partial \Omega_h} 
\\ \nonumber
&\qquad \qquad - (u\circ p_h - T_k(u),n_h\cdot \nabla v)_{\partial \Omega_h}
+ \beta h^{-1} (u \circ p_h - T_k(u), v)_{\partial \Omega_h} 
\end{align}
for all $v\in V_h$. Rearranging the terms we arrive at
\begin{align}
\nonumber
&(\nabla u,\nabla v)_{\Omega_h} 
-(n_h\cdot\nabla u,v)_{\partial \Omega_h}
\\ \nonumber
&\qquad \qquad 
- (T_k(u),n_h\cdot \nabla v)_{\partial \Omega_h}
+ \beta h^{-1} (T_k(u), v)_{\partial \Omega_h} 
\\ \nonumber
& \qquad \qquad +(f + \Delta u, v)_{\Omega_h \setminus \Omega}
\\ \nonumber
&\qquad \qquad - (u\circ p_h - T_k(u),n_h\cdot \nabla v)_{\partial \Omega_h}
+ \beta h^{-1} (u \circ p_h - T_k(u), v)_{\partial \Omega_h} 
\\ \label{eq:derivation-d}
&\qquad=(f,v)_{\Omega_h}  
- (g\circ p_h,n_h\cdot \nabla v)_{\partial \Omega_h}
+ \beta h^{-1}( g\circ p_h, v)_{\partial \Omega_h} 
\end{align}
for all $v\in V_h$. The discrete method is obtained from this
formulation by dropping the consistency terms of highest order,
i.e. those on lines three and four of (\ref{eq:derivation-d}).

\paragraph{Bilinear Forms.}
We define the forms
\begin{align}\label{eq:a0}
a_0(v,w) &:= (\nabla v,\nabla w)_{\Omega_h} 
\\ \nonumber
&\qquad - (n_h\cdot \nabla v,w)_{\partial \Omega_h} 
- (T_k(v),n_h\cdot \nabla w)_{\partial \Omega_h} 
\\ \nonumber
&\qquad + \beta h^{-1}(T_k(v),w)_{\partial \Omega_h} 
\\ \label{eq:ah}
a_h(v,w)& := a_0(v,w) + j_h(v,w)
\\ \label{eq:jh}
j_{h}(v,w) &:= \gamma_j\sum_{F \in \mcF_{h}} 
\sum_{l=1}^p h^{2l-1} ([D_{n_F}^l v],[D_{n_F}^l w])_F
\\ \label{eq:lh}
l_h(w) &:= (f,w)_{\Omega_h} 
- (g\circ p_h ,n_h\cdot \nabla w)_{\partial \Omega_h} 
+ \beta h^{-1}(g\circ p_h, w)_{\partial \Omega_h}                        
\end{align}
where $\gamma_j$ is a positive constant. Here we used the notation:
\begin{itemize}
\item $\mcF_{h}$ is the set of all internal faces to 
elements $K \in \mcK_{h}$ that intersect the set  
$\Omega \setminus \Omega_h \cup \partial \Omega_{h}$, and 
$n_F$ is a fixed unit normal to $F\in \mcF_h$.
\item $D_{n_F}^l$ is the partial derivative of order 
$l$ in the direction of the normal $n_F$ to the face 
$F\in \mcF_{h}$.
\item $[v]|_F = v^+_F - v^-_F$, with $v_F^{\pm} 
= \lim_{s \rightarrow 0^+} 
v(x \mp s n_F)$, is the jump of a discontinuous function 
$v$ across a face $F\in \mcF_{h}$.
\end{itemize}

\paragraph{The Method.} Find: $u_h \in V_h$ such that 
\begin{equation}\label{eq:fem}
a_h(u_h, v) = l_h(v),\qquad \forall v \in V_h
\end{equation}
where $a_h$ is defined in (\ref{eq:ah}) and $l_h$ in (\ref{eq:lh}).

\paragraph{Symmetric Formulation in the Case $\boldsymbol{k=1}$.} 
Using one term in the Taylor expansion gives the following 
forms 
\begin{align}
a_h(v,w) &=(\nabla v,\nabla w)_{\Omega_h} + j_h(v,w)\label{nonsym1}
\\ \nonumber 
&\qquad 
-(n_h \cdot\nabla v,w)_{\partial \Omega_h}
-(v,n_h \cdot\nabla w)_{\partial \Omega_h}
\\ \nonumber
&\qquad  
- (\varrho_h\, \nu_h\cdot \nabla v,n_h \cdot \nabla w)_{\partial \Omega_h}
\\ \nonumber
&\qquad + \beta h^{-1} (T_1(v), w)_{\partial \Omega_h} 
\\
l_h(w)&=(f,w)_{\Omega_h} 
- (g\circ p_h,n_h\cdot \nabla w)_{\partial \Omega_h}  
+ \beta h^{-1}(g\circ p_h, w)_{\partial \Omega_h}  \label{nonsym2}                      
\end{align}
We see that only the terms of the third and the fourth lines of
(\ref{nonsym1}) violate
the symmetry of the formulation. To make it symmetric we choose
$\nu_h:= n_h$, assuming that the discrete approximation $\Omega_h$ is
such that this is a valid choice and also symmetrize the
penalty term in the fourth line by replacing $w$ in the right hand
slot by $T_1(w)$. A similar perturbation is added to the right hand
side to keep consistency. The forms of the resulting symmetric formulation read
\begin{align}
a_h(v,w) &=(\nabla v,\nabla w)_{\Omega_h}+ j_h(v,w) \label{sym1}
\\ \nonumber 
&\qquad 
-(n_h \cdot\nabla v,w)_{\partial \Omega_h}
-(v,n_h \cdot\nabla w)_{\partial \Omega_h}
\\ \nonumber
&\qquad  
- (\varrho_h\, n_h\cdot \nabla v,n_h \cdot \nabla w)_{\partial \Omega_h}
\\ \nonumber
&\qquad + \beta h^{-1} (T_1(v), T_1(w))_{\partial \Omega_h} 
\\
l_h(w)&=(f,w)_{\Omega_h} 
- (g\circ p_h,n_h\cdot \nabla w)_{\partial \Omega_h}  
+ \beta h^{-1}(g\circ p_h, T_1(w))_{\partial \Omega_h}  \label{sym2}                      
\end{align}
The analysis presented below covers this important special case. Also
observe that if more terms are included in the Taylor development
the resulting nonsymmetric part of the matrix is expected to be small, relative to the symmetric part, and the reduced 
symmetric form is likely to be a good preconditioner. 
\section{A Priori Error Estimates}

\subsection{The Energy Norm}
Let the energy norm be defined by
\begin{align}
\tn v \tn_h^2 &= 
\|\nabla v\|^2_{\Omega_h} 
+  \tn v \tn_{j_{h}}^2
+ h\|n_h\cdot \nabla v \|^2_{\partial \Omega_h} 
+ h^{-1}\|v\|^2_{\partial \Omega_h}
\end{align}
where
\begin{equation}
\tn v \tn_{j_h}^2 = j_h(v,v)
\end{equation}

\subsection{Consistency}

In view of (\ref{eq:derivation-d}) we obtain the identity 
\begin{align}\label{eq:consistency-a}
a_h(u-u_h,v) &= (u\circ p_h - T_k(u),n_h\cdot \nabla v)_{\partial \Omega_h}
- \beta h^{-1} (u \circ p_h - T_k(u), v)_{\partial \Omega_h}
\\ \nonumber
&\qquad + (f + \Delta u, v)_{\Omega_h\setminus\Omega}, \qquad \forall v\in V_h
\end{align}
and thus we conclude that 
\begin{align}\label{eq:consistency-b}
|a_h(u-u_h,v)| &\leq \| u\circ p_h - T_k(u)\|_{\partial \Omega_h}
\Big( \|n_h\cdot \nabla v \|_{\partial \Omega_h} 
+  h^{-1} \beta \|v \|_{\partial \Omega_h} \Big)
\\ \nonumber
&\qquad + \|f+ \Delta u \|_{\Omega_h\setminus \Omega} \|v 
\|_{\Omega_h \setminus \Omega}
\\
&\leq h^{-1/2}\| u\circ p_h - T_k(u)\|_{\partial \Omega_h}
\tn v \tn_h 
\\ \nonumber
&\qquad + \|f+ \Delta u \|_{\Omega_h\setminus \Omega} \|v 
\|_{\Omega_h \setminus \Omega},
\qquad \forall v \in V_h
\end{align} 
\paragraph{Estimate of the Error in the Taylor Approximation.}
The Taylor polynomial $T_k(u)(x)$ provides an approximation 
of $u\circ p_h(x)$ and we have the error estimate 
\begin{align}\label{eq:consistency-c}
|v \circ p_h(x) - T_k(v)(x)| 
&\lesssim \left|\int_0^{\varrho_h(x)}  D^{k+1}_{\nu_h} v(x(s))(\varrho_h(x)-s)^{k} ds\right|
\\ \label{eq:consistency-d}
&\lesssim \|D^{k+1}_{\nu_h} v \|_{I_x} \| (\varrho_h(x) - s)^k \|_{I_x}
\\ \label{eq:consistency-e}
&\lesssim \|D^{k+1}_{\nu_h} v \|_{I_x} |\varrho_h(x)|^{k+1/2} 
\end{align}
where $I_x$ is the line segment between $x$ and $p_h(x)$. Combining 
(\ref{eq:consistency-b}) and (\ref{eq:consistency-e}) and recalling
the assumption (\ref{eq:card_hyp}) we arrive 
at the estimate 
\begin{align}
\| v \circ p - T_k(v)\|_{\partial \Omega_h}^2
&\lesssim 
\int_{\partial \Omega_h} 
\|D^{k+1}_{\nu_h}v \|^2_{I_x} |\varrho_h(x)|^{2k+1} dx 
\\
&\lesssim 
\int_{\partial \Omega_h} 
\|D^{k+1}_{\nu_h} v \|^2_{I_{\delta_h}} |\varrho_h(x)|^{2k+1} dx 
\\  \label{eq:consistency-f}
&\lesssim 
\delta_h^{2k+1}
\| D^{k+1} v\|^2_{U_{\delta_h}(\partial \Omega_h)} 
\end{align}
Here we handled the possible overlap of the contributions from
different polygonal sides of $\partial \Omega_h$ by using the fact
that by assumption (\ref{eq:card_hyp}) such an overlap must have a
finite number of contributions uniformly in $h$ and by dropping the
directional derivative, effectively including the derivatives of order
$k+1$ in all directions.

With slightly stronger control of the regularity we obtain the 
estimate 
\begin{align} \label{eq:consistency-g}
\| v \circ p - T_k(v)\|_{\partial \Omega_h}
\lesssim 
\delta_h^{k+1}
\sup_{0\leq t \leq \delta_0} \| D^{k+1} v\|_{L^2(\partial \Omega_t)}
\end{align}
where $\partial \Omega_t = \{x \in \Omega : \rho(x) = t \}$ is 
the levelset with distance $t$ to the boundary $\partial \Omega$.
\paragraph{Estimate of the Residual on {$\boldsymbol{\Omega_h \setminus \Omega}$}.} Suppose that 
\begin{equation}\label{eq:residualregularity}
f+ \Delta u \in H^l(U_{\delta_0}(\Omega))
\end{equation}
which, in view of (\ref{eq:ellipticregularity}) and 
(\ref{eq:extensionstability}), holds 
if $f\in H^{l}(\Omega)$. Using (\ref{eq:residualregularity}) 
and the fact that $f+ \Delta u = 0$ in $\Omega$, 
we obtain the estimate 
\begin{equation} \label{eq:consistency-h}
\|f + \Delta u \|_{\Omega_h \setminus \Omega} 
\lesssim 
\delta^l_h \|D_n^l(f + \Delta u)\|_{\Omega_h \setminus \Omega}
\lesssim 
\delta^{l+1/2}_h 
\sup_{0\leq t \leq \delta_0} 
\|D_n^l(f + \Delta u)\|_{\partial\Omega_t}
\end{equation}
where we used the fact that 
$\Omega_h \setminus \Omega \subset U_\delta(\partial \Omega)$, 
where $\delta \sim \delta_h$.

\paragraph{Estimates of the Consistency Error.}
Combining (\ref{eq:consistency-g}), 
(\ref{eq:consistency-h}), and (\ref{eq:consistency-hh}), 
we obtain the estimate
\begin{align}\label{eq:consistency-i}
|a_h(u-u_h,v)| &\leq 
\delta^{k+1}_h 
\sup_{0 \leq t \leq \delta_0} \| D^{k+1} u\|_{L^2(\partial \Omega_t)}
\Big(\|n_h\cdot \nabla v \|_{\partial \Omega_h} 
+ h^{-1}\beta \|v \|_{\partial \Omega_h} \Big)
\\ \nonumber
&\qquad + 
\delta^{l+1/2}_h \sup_{0 \leq t \leq \delta_0} 
\|D_n^l(f+ \Delta u) \|_{\partial\Omega_t} 
\| v \|_{\Omega_h\setminus \Omega},
 \qquad \forall v \in V_h
\end{align} 
This estimate will be used when we derive an $L^2$ estimate of 
the error while for the energy error estimate we continue the 
estimation using the bound
\begin{equation} \label{eq:consistency-hh}
\| v \|_{\Omega_h \setminus \Omega} 
\lesssim \delta_h^{1/2} \tn v \tn_h
\end{equation}
which leads to
\begin{align}\label{eq:consistency-ii}
|a_h(u-u_h,v)| &\leq 
\Big( h^{-1/2} \delta^{k+1}_h 
\sup_{0\leq t \leq \delta_0} \| D^{k+1} u\|_{L^2(\partial \Omega_t)}
\\ \nonumber
&\qquad \qquad + 
\delta^{l+1}_h \sup_{0\leq t \leq \delta_0} 
 \|D_n^l(f+ \Delta u) \|_{\partial\Omega_t} \Big)\tn v \tn_h, 
 \qquad \forall v \in V_h
\end{align} 

\begin{rem}
We may upper bound the right hand sides further using global trace
inequalities leading to
\begin{equation}\label{eq:trace_reg_bound_a}
\sup_{0\leq t \leq \delta_0} \| D^{k+1} u\|_{L^2(\partial
  \Omega_t)} \lesssim \|u\|_{H^{k+2}(\Omega)} \lesssim \|f \|_{H^k(\Omega)}
\end{equation}
and
\begin{equation}\label{eq:trace_reg_bound_b}
\sup_{0\leq t \leq \delta_0} 
 \|D_n^l(f+ \Delta u) \|_{\partial\Omega_t} \lesssim
 \|f\|_{H^{l+1}(\Omega)}+  \|\Delta u \|_{H^{l+1}(\Omega)}
 \lesssim \|f\|_{H^{l+1}(\Omega)}
\end{equation}
\end{rem}

\subsection{Inverse Inequality}

Using the additional stability provided by the stabilization term 
$j_h$ we have the following inverse inequalities  
\begin{equation}\label{eq:inversec}
\|\nabla v \|^2_{\mcN_h} 
\lesssim 
\|\nabla v \|^2_{\Omega_h} + \tn v \tn_{j_h}^2, \qquad \forall v \in V_h
\end{equation} 
and
\begin{equation}\label{eq:inverseL2}
\| v \|^2_{\mcN_h} 
\lesssim 
\| v \|^2_{\Omega_h} + h^2 \tn v \tn_{j_h}^2, \qquad \forall v \in V_h
\end{equation} 
See \cite{MasLarLog14} for a proof.

\subsection{Coercivity and Continuity} 
 
We have coercivity
\begin{equation}\label{eq:coercivity}
\tn v \tn^2_h \lesssim a_h(v,v),\qquad \forall h \in (0,h_0]
\end{equation}
if $h_0$ small enough and $\beta$ large enough, and continuity 
\begin{equation}\label{eq:continuity}
a_h(v,w) 
\lesssim 
\Big(\tn v \tn_h + h^{-1/2}\|T_{1,k}(v) \|_{\partial \Omega_h} \Big) 
\tn w \tn_h, \qquad \forall v \in V + V_h, w \in V_h
\end{equation}
where 
\begin{equation}
V = H^{k+1/2}(\mcN_h) \cap H^{3/2}(\mcN_h) \cap H^{p+1/2}(\mcN_h)
\end{equation}
is the space on which the functional $V \ni v \mapsto a_h(v,w) \in \IR$, for 
a fixed $w\in V_h$ and fixed $h \in (0,h_0]$ is bounded.
The continuity estimate (\ref{eq:continuity}) follows directly 
from the Cauchy-Schwarz inequality and we next verify the coercivity 
estimate (\ref{eq:coercivity}). 

\paragraph{Verification of (\ref{eq:coercivity}).} Using the 
notation 
\begin{equation}
T_{1,k}(v) = T_k(v) - v
\end{equation}
we obtain
\begin{align}
a_h(v,v) 
&= (\nabla v, \nabla v)_{\Omega_h} + j_h(v,v)
- 2(n_h\cdot \nabla v,v)_{\partial \Omega_h}  
+ \beta h^{-1} (v,v)_{\partial \Omega_h}
\\ \nonumber
&\qquad + \beta h^{-1}(T_{1,k}(v),v)_{\partial \Omega_h} 
- (T_{1,k}(v),n_h\cdot \nabla v)_{\partial \Omega_h}
\\
&\gtrsim 
\|\nabla v\|^2_{\Omega_h} + \tn v\tn^2_{j_h}
- h^{1/2} \|n_h \cdot \nabla v\|_{\partial \Omega_h} 
     h^{-1/2} \|v\|_{\partial \Omega_h}  
+ \beta h^{-1}\|v\|^2_{\partial \Omega_h}
\\ \nonumber
&\qquad 
- \beta h^{-1/2}\|T_{1,k}(v)\|_{\partial \Omega_h}
h^{-1/2}\|v\|_{\partial \Omega_h} 
- h^{-1/2}\|T_{1,k}(v)\|_{\partial \Omega_h}
h^{1/2}\|n_h\cdot \nabla v\|_{\partial \Omega_h}
%
%
\end{align}
Now we have the inverse bounds
\begin{equation}\label{eq:inversea}
h^{1/2}\| n_h \cdot \nabla v \|_{\partial \Omega_h} 
\lesssim 
\|\nabla v \|_{\mcN_h(\partial \Omega_h)}
\end{equation}
\begin{align}
h^{-1/2}\|T_{1,k}(v)\|_{\partial \Omega_h} 
&\lesssim 
h^{-1}\|T_{1,k}(v)\|_{\mcN_h(\partial \Omega_h)}
\\
&\lesssim 
\underbrace{\left( \sum_{j=1}^k \frac{\delta_h^j}{h^j} \right)}_{\lesssim \gamma(h) \sim h^{-1}o(h)} 
\|\nabla v\|_{\mcN_h(\partial \Omega_h)} 
\\ \label{eq:inverseb}
&\lesssim \gamma(h) \|\nabla v\|_{\mcN_h(\partial \Omega_h)}
\end{align}
where $\gamma(h)\rightarrow 0$ as $h\rightarrow 0$. Using (\ref{eq:inversec}) together with obvious estimates these 
bounds conclude the proof of the coercivity result 
(\ref{eq:coercivity}) for $\beta$ large enough and 
$h\in(0,h_0]$, with $h_0$ small enough.

{\hfill \mbox{\fbox{} } \newline}


\subsection{Interpolation Estimates}
 
Let 
\begin{equation}\label{eq:ext_interp}
\pi_{h}: H^1(\Omega) \ni u \mapsto \pi_{SZ,h} Eu \in V_{h}
\end{equation}
where $E$ is the extension operator introduced in Section \ref{sec:extension}, and $\pi_{SZ,h}$ is the Scott-Zhang 
interpolation operator. The following error estimate for the Scott-Zhang interpolant is well
known \cite{SZ90}
\begin{equation}\label{eq:interpolationestimate}
\|u - \pi_{SZ,h} u \|_{H^m(K)} 
\lesssim h^{s-m}\|u\|_{H^s(\mcN_h(K))},
\qquad 
0\leq m \leq s \leq p+1, 
\qquad
K \in \mcK_h
\end{equation}
Using the properties of the extension operator we then immediately deduce this
interpolation error estimate for (\ref{eq:ext_interp})
\begin{equation}\label{eq:interpolationestimatedg}
\tn u - \pi_h u \tn_h 
+ h^{-1/2}\|T_{1,k} (u - \pi_h u) \|_{\partial \Omega_h} 
\lesssim 
h^p \| u \|_{H^{p+1}(\Omega)} 
\end{equation}
\paragraph{Verification of (\ref{eq:interpolationestimatedg}).}
The first term in (\ref{eq:interpolationestimatedg}) is estimated 
using the trace inequality 
\begin{equation}\label{eq:trace}
\| v \|^2_{\partial \Omega_h \cap K} 
\lesssim h^{-1}\| v \|^2_K + h \|\nabla v \|_K^2, \qquad K \in \mcK_{h} 
\end{equation}
see \cite{HaHaLa03}, followed by the interpolation 
estimate (\ref{eq:interpolationestimate}) and stability 
of the extension operator (\ref{eq:extensionstability}). Again 
using the trace inequality (\ref{eq:trace}) the second term in 
(\ref{eq:interpolationestimatedg}) can be estimated as follows
\begin{align}
h^{-1/2}\|T_{1,k} (u - \pi_h u) \|_{\partial \Omega_h}
&\lesssim 
h^{-1}\|T_{1,k} (u - \pi_h u) \|_{\mcN_h(\partial \Omega_h)} 
+ 
\|\nabla T_{1,k} (u - \pi_h u) \|_{\mcN_h(\partial \Omega_h)}
\\
&\lesssim h^{p} \|u \|_{H^{p+1}(\Omega)} 
\end{align}
where finally we used the fact that $\delta_h \lesssim h$ and 
the estimate
\begin{align}
h^{m-1}\| \nabla^m T_{1,k} (u - \pi_h u) \|_K 
&\lesssim 
\sum_{j=1}^k \delta_h^j h^{m-1} \|(u - \pi_h u )\|_{H^{j+m}(K)} 
\\
&\lesssim 
\sum_{j=1}^k h^j h^{m-1} h^{p+1-(j+m)} \|u \|_{H^{p+1}(\mcN(K))} 
\\
&\lesssim 
h^{p}\|u \|_{H^{p+1}(\mcN(K))} 
\end{align}
for $m=0,1$ and $K \in \mcK_h(\partial \Omega_h)$. 
{\hfill \mbox{\fbox{} } \newline}
 
\subsection{Error Estimates}

\begin{thm}\label{thm:energy}
If $\delta_h = o(h)$, then the following estimate holds
\begin{align}\label{eq:errorestenergy}
\tn u - u_h \tn_{h} 
&\lesssim 
h^{p} \|u \|_{H^{p+1}(\Omega)}
+ 
h^{-1/2} \delta_h^{k+1} 
\sup_{0\leq t \leq \delta_0} \| D^{k+1} u\|_{L^2(\partial \Omega_t)}
\\ \nonumber
&\qquad + 
\delta_h^{l+1} 
\sup_{-\delta_0\leq t < 0} 
\| D_n^{l} (f + \Delta u)\|_{L^2(\partial \Omega_t)}
\end{align}
\end{thm}

\begin{table}[h]\label{table:energy}
\begin{center}
\begin{tabular}{|c|c||c|c||c|c|}
\hline 
$p$ & $h^p$ & $k$ & $h^{-1/2}\delta_h^{k+1}$& l & $\delta_h^{l+1}$ 
\\ 
\hline 
1 & $h^1$ & 0 & $h^{1.5}$ &  0 &  $h^2$
\\ 
2 & $h^2$ & 1 & $h^{3.5}$ & 1 & $h^4$ 
\\ 
3 & $h^3$ & 2 & $h^{5.5}$ & 2 & $h^6$ 
\\ 
4 & $h^4$ & 3 & $h^{7.5}$ & 3 & $h^8$
\\ 
\hline 
\end{tabular}
\caption{The order of the terms in the energy error estimate under the 
assumption $\delta_h \lesssim h^2$. We conclude that we obtain optimal order of convergence for $p=2,3,$ with one term, $k=1$, in the Taylor expansion and for $p=4,5,$ with two terms, $k=2$.}
\end{center}
\end{table}

\begin{proof} We first note that adding and subtracting an interpolant 
and using the triangle inequality and the interpolation estimate (\ref{eq:interpolationestimatedg}), we obtain
\begin{align}
\tn u -  u_h \tn_h 
&\lesssim 
\tn u - \pi_h u \tn_h + \tn \pi_h u - u_h \tn_h 
\\ \label{proof:energya}
&\lesssim h^{p} \| u \|_{H^{p+1}(\Omega)} + \tn \pi_h u - u_h \tn_h
\end{align}
For the second term on the right hand side we have the estimates
\begin{align}
\tn \pi_h u - u_h \tn_h^2 
&\lesssim a_h(\pi_h u - u_h,\pi_h u - u_h) 
\\
&= a_h(\pi_h u - u,\pi_h u - u_h) + a_h(u- u_h,\pi_h u - u_h)
\\
&\lesssim 
\Big( \tn \pi_h u - u \tn_h  
+ h^{-1/2}\|T_{1,k}(\pi_h u - u)\|_{\partial \Omega_h}\Big) 
\tn \pi_h u - u_h \tn_h
\\ \nonumber
&\qquad + h^{-1/2}\|u\circ p_h - T_k(u) \|_{\partial \Omega_h} 
\tn \pi_h u - u_h \tn_h
\\ \nonumber
&\qquad + \| f + \Delta u \|_{\Omega_h\setminus \Omega_h} 
\| \pi_h u - u_h \|_{\Omega_h \setminus \Omega}
\\
&\lesssim \label{proof:energyb}
h^{p} \| u \|_{H^{p+1}(\Omega)} 
\tn \pi_h u - u_h \tn_h
\\ \nonumber
&\qquad + 
 h^{-1/2} \delta_h^{k+1} \sup_{0\leq t 
\leq \delta_0} \| D^{k+1} u\|_{L^2(\partial \Omega_t)}
\tn \pi_h u -  u_h \tn_h
\\ \nonumber
&\qquad + 
\delta_h^{l+1} 
\sup_{-\delta_0\leq t < 0} 
\| D_n^{l} (f + \Delta u)\|_{L^2(\partial \Omega_t)}
\tn \pi_h u -  u_h \tn_h
\end{align}
where we used coercivity (\ref{eq:coercivity}), 
added and subtracted the exact solution $u$, 
estimated the first term using continuity 
(\ref{eq:continuity}) followed by the interpolation 
estimate (\ref{eq:interpolationestimatedg}) and the 
second using the consistency estimate (\ref{eq:consistency-c}). 
Combining estimates (\ref{proof:energya}) and (\ref{proof:energyb}) 
concludes the proof.
\end{proof}

\begin{thm}\label{thm:L2}
If $\delta_h \lesssim h^2$, then the following estimate 
holds
\begin{align}\label{eq:errorestL2}
\|e\|_{\Omega_h} 
&\lesssim h^{p+1} \| u \|_{H^{p+1}(\Omega)}
\\ \nonumber 
&\qquad + \delta^{k+1}_h 
\sup_{0 \leq t \leq \delta_0} \| D^{k+1} u\|_{L^2(\partial \Omega_t)}
\\ \nonumber 
&\qquad + \delta_h^{l+3/2} 
\sup_{0 \leq t \leq \delta_0} 
\|D_n^l(f+ \Delta u) \|_{\partial\Omega_t}
\end{align}
\end{thm}

\begin{table}[h]\label{table:L2}
\begin{center}
\begin{tabular}{|c|c||c|c||c|c|}
\hline 
$p$ & $h^{p+1}$ & $k$ &  $\delta_h^{k+1}$ & l & $\delta_h^{l+3/2}$ 
\\ 
\hline 
1 & $h^2$ & 0 &  $h^{2}$ &  0 &  $h^3$
\\ 
2 & $h^3$ & 1 & $h^{4}$ & 1 & $h^5$ 
\\ 
3 & $h^4$ & 2 & $h^{6}$ & 2 & $h^7$ 
\\ 
4 & $h^5$ & 3 & $h^{8}$ & 3 & $h^9$
\\ 
\hline 
\end{tabular}
\caption{The order of the terms in the energy error estimate under the 
assumption that $\delta_h \lesssim h^2$. We conclude that we obtain optimal order of convergence for $p=2,3,$ with one term, $k=1$, in the Taylor expansion and for $p=4,5,$ with two terms, $k=2$.}
\end{center}
\end{table}

\begin{proof} Let $\phi \in H^1_0(\Omega)$ be the solution 
to the dual problem
\begin{equation}\label{eq:dual}
a(v,\phi) = (v,\psi), \qquad v \in H^1_0(\Omega)
\end{equation}
where $\psi=u-u_h$ on $\Omega_h$ and $\psi=0$ on 
$\Omega \setminus \Omega_h$, and extend $\phi$ using 
the extension operator to $U_{\delta_0}(\Omega)$. Then 
we have the stability estimate
\begin{equation}\label{eq:dualstab}
\| \phi \|_{H^2(\Omega)} \lesssim 
\| \psi \|_{\Omega \cap \Omega_h}
\end{equation}

We obtain the following representation formula for the error 
\begin{align}
\|e\|^2_{\Omega_h} &= (e,\psi+\Delta \phi)_{\Omega_h} 
- (e,\Delta \phi)_{\Omega_h}
\\
&=(e,\psi +\Delta \phi)_{\Omega_h \setminus \Omega}
+ (\nabla e,\nabla \phi )_{\Omega_h} 
- (e,n_h\cdot \nabla \phi)_{\partial \Omega_h}
\\
&=(e,\psi +\Delta \phi)_{\Omega_h\setminus \Omega}
+ a_0(e,\phi) + b_h(e,\phi)
\\
&=I + II + III
\end{align}
where 
\begin{align}
III&= (T_k(e) - e,n_h \cdot \nabla \phi)_{\partial \Omega_h} 
- \beta h^{-1} (T_k(e),\phi)_{\partial \Omega_h} 
+ (n_h \cdot \nabla e, \phi )_{\partial \Omega_h}
\\
&= (T_{1,k}(e),n \cdot \nabla \phi)_{\partial \Omega_h} 
- \beta h^{-1} (e,\phi)_{\partial \Omega_h} 
\\ \nonumber
& \qquad
- \beta h^{-1} (T_{1,k}(e),\phi)_{\partial \Omega_h}
+ (n_h \cdot \nabla e, \phi )_{\partial \Omega_h}
\end{align}
\paragraph{Term $\bfI$.} We have 
\begin{align}
|I|&=|(e,\psi+\Delta \phi)_{\Omega_h \setminus \Omega}|
\\
&\lesssim \|e\|_{\Omega_h \setminus \Omega} 
\|\psi+\Delta \phi\|_{\Omega_h \setminus \Omega}
\\
&\lesssim 
\Big(\delta_h^2 \|n\cdot \nabla e \|^2_{\Omega_h \setminus \Omega}
+ \delta_h \|e\|^2_{\partial \Omega_h} \Big)^{1/2}
\Big(\| \psi \|_{\Omega_h \setminus \Omega} +  
\| \Delta \phi \|_{\Omega_h\setminus \Omega} \Big)
\\
&\lesssim 
\Big( (\delta_h^2 + h \delta_h) \tn e \tn_h^2 \Big)^{1/2}
\Big( \|e\|_{\Omega_h\setminus \Omega} + \| \phi \|_{H^2(\Omega)} \Big)
\\ \label{est:L2I}
&\lesssim 
\underbrace{(h^{-2} \delta_h + h^{-1}\delta_h)^{1/2}}_{\lesssim 1} h \tn e \tn_h \| e \|_{\Omega_h}
\end{align}
Here we used the estimate 
\begin{equation}\label{eq:poincarespecial}
\| v \|^2_{\Omega_h \setminus \Omega}
\lesssim 
\delta_h^2 \| n\cdot \nabla v \|^2_{\Omega_h \setminus \Omega}
+
\delta_h \| v \|^2_{\partial \Omega_h}, \qquad v \in H^1(\Omega_h)
\end{equation}
with $v=e$, the definition of the energy norm to conclude that 
$h^{-1}\|e\|^2_{\partial \Omega_h} \lesssim \tn e \tn_h^2$, the 
stability (\ref{eq:extensionstability}) of the extension operator, 
the stability (\ref{eq:dualstab}) of the dual problem and the 
assumption that $\delta_h \lesssim h^{2}$.

\paragraph{Term $\bfI\bfI$.} Adding and subtracting an interpolant 
we obtain
\begin{align}
|II| &= |a_h(e,\phi - \pi_h \phi) + a_h(e,\pi_h\phi)|
\\
&\lesssim 
\tn e \tn_h \tn \phi - \pi_h \phi \tn_h 
+
|a_h(e,\pi_h\phi)|
\\
&\lesssim 
h \tn e \tn_h \|\phi\|_{H^2(\Omega)}
+ |a_h(e,\pi_h\phi)|
\\ \label{proof:L2a}
&\lesssim 
h \tn e \tn_h \| e\|_{\Omega_h}
+ |a_h(e,\pi_h\phi)|
\end{align}
To estimate the second term on the right hand side we employ (\ref{eq:consistency-i}), with $v= \pi_h \phi$,
\begin{align}\label{proof:L2b}
|a_h(e,\pi_h \phi)| &\leq 
\delta^{k+1}_h 
\sup_{0 \leq t \leq \delta_0} \| D^{k+1} u\|_{L^2(\partial \Omega_t)}
\Big(\|n_h\cdot \nabla \pi_h \phi  \|_{\partial \Omega_h} 
+ h^{-1}\beta \|\pi_h \phi \|_{\partial \Omega_h} \Big)
\\ \nonumber
&\qquad + 
\delta_h^{l+1/2} \sup_{0 \leq t \leq \delta_0} 
\|D_n^l(f+ \Delta u) \|_{\partial\Omega_t} 
\| \pi_h \phi \|_{\Omega_h\setminus \Omega}
\end{align} 
Here we have the estimates
\begin{align}
\|n_h \cdot \nabla \pi \phi \|_{\partial \Omega_h} 
+ 
h^{-1}\|\pi \phi\|_{\partial \Omega_h}
&\lesssim \label{proof:L2c}
\|n_h \cdot \nabla (\pi \phi - \phi) \|_{\partial \Omega_h} 
+ 
h^{-1}\|\pi \phi - \phi \|_{\partial \Omega_h}
\\ \nonumber
&\qquad 
+ \|n_h \cdot \nabla \phi \|_{\partial \Omega_h} 
+ 
h^{-1}\|\phi\|_{\partial \Omega_h}
\\
&\lesssim
h^{-1/2}\tn \pi \phi - \phi \tn_h
\\ \nonumber 
&\qquad +
\|\phi \|_{H^2(\Omega_h)}
+
h^{-1} \delta_h^{1/2} \|\phi \|_{H^1(U_{\delta_h(\partial \Omega)})}  
\\
&\lesssim
h^{1/2}\| \phi \|_{H^2(\Omega)}
+
\|\phi \|_{H^2(\Omega_h)}
+
h^{-1} \delta_h^{1/2} \|\phi \|_{H^1(U_{\delta_h(\partial \Omega)})}
\\
&\lesssim
(h^{1/2} + 1 + h^{-1}\delta_h^{1/2})\| e \|_{\Omega_h}
\\
&\lesssim
\| e \|_{\Omega_h}
\end{align}
and 
\begin{align}
\| \pi_h \phi \|_{\Omega_h\setminus \Omega} 
&
\leq \label{proof:L2d}
\| \pi_h \phi - \phi \|_{\Omega_h\setminus \Omega}
+
\| \phi \|_{\Omega_h\setminus \Omega}
\\
&\lesssim 
h^{2}\| \phi \|_{H^2(\Omega)} + 
\delta_h \|  \nabla \phi \|_{U_{\delta_h}(\partial \Omega)}
\\
& \lesssim 
(h^{2} + \delta_h) \|e \|_{\Omega_h}
\\
& \lesssim 
\delta_h \|e \|_{\Omega_h}
\end{align}
where, in both estimates, we used the assumption 
$\delta_h \lesssim h^2$, as well as the following bounds
\begin{equation}\label{eq:technicalA}
\|\phi \|_{\partial \Omega_h} 
\lesssim \delta_h^{1/2} 
\| n\cdot \nabla \phi \|_{U_{\delta_h}(\partial \Omega)}
\end{equation}
\begin{equation}\label{eq:technicalC}
\|\phi \|_{\Omega_h\setminus \Omega} 
\lesssim 
\|\phi \|_{U_{\delta_h}(\partial \Omega)} 
\lesssim \delta_h 
\| n\cdot \nabla \phi \|_{U_{\delta_h}(\partial \Omega)}
\end{equation}
see the Appendix for the proof of these estimates.
Combining estimates (\ref{proof:L2b}), (\ref{proof:L2c}), and 
(\ref{proof:L2d}), we arrive at 
\begin{align}\label{proof:L2e}
|a_h(e,\pi_h \phi)| 
&\lesssim \Big(
\delta^{k+1}_h 
\sup_{0 \leq t \leq \delta_0} \| D^{k+1} u\|_{L^2(\partial \Omega_t)}
\\ \nonumber 
& \qquad + \delta_h^{l+3/2} 
\sup_{0 \leq t \leq \delta_0} 
\|D_n^l(f+ \Delta u) \|_{\partial\Omega_t} \Big) \|e \|_{\Omega_h}
\end{align}
which together with (\ref{proof:L2a}) gives
\begin{align}\label{est:L2II}
|II| 
&\lesssim \Big( h \tn e \tn_h 
+ 
 \delta^{k+1}_h 
\sup_{0 \leq t \leq \delta_0} \| D^{k+1} u\|_{L^2(\partial \Omega_t)}
\\ \nonumber 
& \qquad + \delta_h^{l+3/2} 
\sup_{0 \leq t \leq \delta_0} 
\|D_n^l(f+ \Delta u) \|_{\partial\Omega_t}
  \Big) \|e \|_{\Omega_h}
\end{align}

\paragraph{Term $\bfI\bfI\bfI$.}
Using the Cauchy-Schwarz inequality we get
\begin{align}
|III| &= |b_h(e,\phi)|
\\
&\lesssim \|T_{1,k}(e)\|_{\partial \Omega_h} 
\|n_h \cdot \nabla \phi\|_{\partial \Omega_h} 
+ \beta h^{-1} \|e\|_{\partial \Omega_h} \|\phi\|_{\partial \Omega_h} 
\\ \nonumber
&\qquad + \beta h^{-1} \|T_{1,k}(e)\|_{\partial \Omega_h} \|\phi\|_{\partial \Omega_h}
+ \|n_h \cdot \nabla e\|_{\partial \Omega_h} \|\phi\|_{\partial \Omega_h}
\\
&\lesssim \|T_{1,k}(e)\|_{\partial \Omega_h} 
\Big( h^{-1}\|\phi\|_{\partial \Omega_h} +  \|n_h \cdot \nabla \phi\|_{\partial \Omega_h} \Big)
\\ \nonumber
&\qquad + \tn e \tn_h h^{-1/2}\| \phi \|_{\partial \Omega_h}
\\
&\lesssim \Big( \|T_{1,k}(e)\|_{\partial \Omega_h} 
+ h^{-1/2} \delta_h \tn e \tn_h \Big) \| e \|_{\Omega_h}
\\ \label{est:L2III}
&\lesssim \Big( h^{p+1} \| u \|_{H^{p+1}(\Omega)} + 
(h^{-3/2}\delta_h) h \tn e \tn_h \Big) \|e\|_{\Omega_h}
\end{align}
where we used (\ref{eq:technicalA}) and (\ref{eq:technicalC}) 
followed by the stability estimate for the dual problem (\ref{eq:dualstab}), and at last the estimate 
\begin{equation}\label{eq:IIIa}
\|T_{1,k}(e)\|_{\partial \Omega_h}
\lesssim 
h^{p+1} \| u \|_{H^{p+1}(\Omega)} + (h^{-3/2}\delta_h) h \tn e \tn_h
\end{equation}

\paragraph{Verification of (\ref{eq:IIIa}).} We have
\begin{align}
\|T_{1,k}(e)\|_{\partial\Omega_h} 
&\lesssim 
\sum_{j=1}^k \delta_h^j \| D_{\nu_h}^j e \|_{\partial \Omega_h}
\end{align}
and for each of the terms $\| D_{\nu_h}^j e \|_{\partial \Omega_h}$, $j=1,\dots, k$, we obtain by adding and subtracting an interpolant, using the interpolation estimate (\ref{eq:interpolationestimate}) for the first 
term and an inverse estimate for the second, the estimates
\begin{align}
\| D_{\nu_h}^j e \|^2_{\partial \Omega_h} 
&\lesssim 
h^{-1}\| D_{\nu_h}^j e \|^2_{\mcN_h(\partial \Omega_h)}
+
h \| \nabla D_{\nu_h}^j e \|^2_{\mcN_h(\partial \Omega_h)}
\\
&\lesssim 
h^{-1}\| D_{\nu_h}^j (u-\pi_h u) \|^2_{\mcN_h(\partial \Omega_h)}
+
h \| \nabla D_{\nu_h}^j (u - \pi_h u) \|^2_{\mcN_h(\partial \Omega_h)}
\\ \nonumber
& \qquad
+ h^{-1}\| D_{\nu_h}^j (\pi_h u - u_h) \|^2_{\mcN_h(\partial \Omega_h)}
+ h \| \nabla D_{\nu_h}^j (\pi_h u - u_h)  \|^2_{\mcN_h(\partial \Omega_h)}
\\
&\lesssim 
h^{2p+1-2j}\| u \|^2_{H^{p+1}(\mcN_h(\mcN_h(\partial \Omega_h)))}
+ h^{1-2j}\| \nabla(\pi_h u - u_h) \|^2_{\mcN_h(\partial \Omega_h)}
\\
&\lesssim 
h^{2p+1-2j}\| u \|^2_{H^{p+1}(\Omega)}
+ h^{1-2j} \| \nabla e \|^2_{\mcN_h(\partial \Omega_h)}
\end{align}
which leads to
\begin{align}
 \delta_h^{2j} \| D_{\nu_h}^j e \|^2_{\partial \Omega_h}
 &\lesssim 
 h^{-1}(\delta_h/h)^{2j} h^{2(p+1)} \| u \|^2_{H^{p+1}(\Omega)}
+   h (\delta_h/h)^{2j} \| \nabla e \|^2_{\mcN_h(\partial \Omega_h)}
\\
&\lesssim 
(h^{-3}\delta_h^2) h^{2(p+1)} \| u \|^2_{H^{p+1}(\Omega)}
+   (h^{-3} \delta_h^2) h^2 \| \nabla e \|^2_{\mcN_h(\partial \Omega_h)}
\end{align}
where we used (\ref{eq:geomassum-a}) and the fact $\delta_h/h^2 
\lesssim 1$. Thus we have
\begin{align}
\|T_{1,k}(e)\|_{\partial\Omega_h} 
&\lesssim 
\sum_{j=1}^k \delta_h^j \| D_{\nu_h}^j e \|_{\partial \Omega_h}
\lesssim (h^{-3/2} \delta_h)
\Big( h^{p+1} \| u \|_{H^{p+1}(\Omega)}
+ h \tn e \tn_h \Big)
\end{align}

\paragraph{Conclusion of the Proof.} Collecting the bounds 
(\ref{est:L2I}), (\ref{est:L2II}), and (\ref{est:L2III}), of 
Terms $I,II,$ and $III$, we obtain
\begin{align}
\|e\|_{\Omega_h} 
&\lesssim h \tn e \tn_h 
\\ \nonumber
&\qquad + h^{p+1} \| u \|_{H^{p+1}(\Omega)}
\\ \nonumber 
&\qquad + \delta^{k+1}_h 
\sup_{0 \leq t \leq \delta_0} \| D^{k+1} u\|_{L^2(\partial \Omega_t)}
\\ \nonumber 
&\qquad + \delta_h^{l+3/2} 
\sup_{0 \leq t \leq \delta_0} 
\|D_n^l(f+ \Delta u) \|_{\partial\Omega_t}
\end{align}
which together with the energy norm error estimate (\ref{eq:errorestenergy}) concludes the proof.
%
%
%
%
\end{proof}

\begin{thm}\label{thm:Omega} The following estimates hold
\begin{equation}
\|\nabla e \|_\Omega \lesssim h^p\| u \|_{H^{p+1}(\Omega)} + \tn e \tn_h
\end {equation}
and 
\begin{equation} 
\qquad \| e \|_{\Omega} \lesssim h^{p+1}\|u \|_{H^{p+1}(\Omega)} + \| e\|_{\Omega_h} + h \tn e \tn_h
\end{equation}
\end{thm}
\begin{proof}
Adding and subtracting an interpolant, using the interpolation 
estimate (\ref{eq:interpolationestimate}), and the inverse inequality 
(\ref{eq:inversec}) or (\ref{eq:inverseL2}), we obtain, for $m=0,1,$
\begin{align}
\| \nabla^m e \|_{\Omega\setminus \Omega_h} 
& \lesssim \| \nabla^m (u - \pi_h u ) \|_{\Omega\setminus \Omega_h} 
+
\| \nabla^m (\pi_h u - u_h) \|_{\Omega\setminus \Omega_h}
 \\
& \lesssim h^{p+1-m} \|u \|_{H^{p+1}(\Omega)} 
+
\| \nabla^m (\pi_h u - u_h) \|_{\Omega_h} 
+ h^{1-m} \tn \pi_h u - u_h \tn_{j_h}
 \\
& \lesssim h^{p+1-m} \|u \|_{H^{p+1}(\Omega)} 
+
\| \nabla^m e  \|_{\Omega_h} 
+ h^{1-m} \tn e \tn_{j_h}
 \\
& \lesssim h^{p+1-m} \|u \|_{H^{p+1}(\Omega)} 
+
\| \nabla^m e  \|_{\Omega_h} 
+ h^{1-m} \tn e \tn_{h}
\end{align} 
which concludes the proof.
\end{proof}

\begin{rem}
If for a given $p$ the lowest values of $k$ and $l$ are chosen so that
optimal convergence is obtained, it is straightforward to use a trace
inequality, see (\ref{eq:trace_reg_bound_a}) and (\ref{eq:trace_reg_bound_b}), to show that 
\[
\|u - u_h\|_{L^2(\Omega_h)} + h \tn u - u_h \tn \lesssim h^{p+1} (\|f\|_{H^{p-1}(\Omega)}+\|u\|_{H^{p+1}(\Omega)})
\]
Therefore the regularities required for optimality of the consistency
error of the boundary approximation are always optimal compared to the
polynomial approximation.
\end{rem}

\begin{rem} We note that we obtain, as a special case, optimal order 
error estimates for the standard cut Nitsche method with 
approximate domains by assuming $k=0$ and 
\begin{equation}
\delta_h \lesssim h^{p+1/2}
\end{equation}
for the energy norm estimate and 
\begin{equation}
\delta_h \lesssim h^{p+1}
\end{equation}
for the $L^2$ norm estimate. The latter assumption is comparable with 
the geometric approximation accuracy achieved by standard isoparametric finite elements of order 
$p$. 
\end{rem}

\section{Numerical Examples}

In the numerical examples, we use implicitly defined boundaries by use of zero isolines to predefined functions.
Two examples have been considered, one with both convex and concave boundaries, so that cut elements can have parts outside the
actual domain, and one example with nonzero boundary conditions where we also compare setting the boundary condition on the exact boundary to setting
 them on computational
boundary. In all examples the stabilization parameters were set to $\gamma_j=1/10$, $\beta=100$. 

\subsection{Convex and Concave Boundaries}

In our first exampe we consider a ring-shaped domain. In Fig. \ref{fig:domain1} we show the zero isoline of the function
$\phi= (R-1/4)(R-3/4)$, $R=\sqrt{x^2+y^2}$, used to implicitly define the domain,
and the resulting mesh after removing the cut part.  
On this ring, we used a load corresponding to the exact solution being  a square function in $R$,
\begin{equation}
u=20(3/4 - R)(R-1/4)\label{exactsol}
\end{equation}
 with zero boundary conditions
on the outside as well as inside boundaries. The elements on the inside of the ring are partially outside the computational domain;
outside the domain the load was extended by zero and the exact solution (in the convergence study) by (\ref{exactsol}).

We show an elevation of the approximate solution on one of the meshes in a sequence in Fig. \ref{fig:elev1}. In Fig. \ref{fig:conv1p2} and \ref{fig:conv1p3} we show the convergence rates obtained using the symmetric method (\ref{sym1})--(\ref{sym2})
for $P^2$ and $P^3$ elements (polynomial orders $p=2$ and $p=3$), respectively. We also show the suboptimal convergence rates of the original Nitsche method. Note in particular that the optimal rate is attained also for $p=3$ even though only the first two terms in the Taylor series are accounted for.

\subsection{Nonzero Boundary Conditions}

The domain for the second exampe lies inside the ellipse defined by the zero isoline to $\phi=x^2/(3/4)^2+y^2/(1/2)^2-1$.
In Fig. \ref{fig:domain2} we show the zero isoline of this function
and the resulting mesh after removing the cut part.  
On this domain we use the right-hand side 
\[
f=\pi^2\cos{(\pi x/2)}\cos{(\pi y/2)}
\]
corresponding to the exact solution $u=\cos{(\pi x/2)}\cos{(\pi y/2)}$. This function also defines the boundary conditions
on the cut boundary. An elevation of an approximate solution on one of the meshes in a sequence is given in Fig. \ref{fig:elev2}.

In Fig. \ref{fig:convell} we show the observed $L_2$ convergence with a $P^3$ approximation using four different approaches:
\begin{itemize}
\item The symmetric method (\ref{sym1})--(\ref{sym2}).
\item The unsymmetric Taylor expansion with two terms.
\item The unsymmetric Taylor expansion with three terms.
\item Prescribing the boundary condition on the cut boundary (using the fact that the exact solution is known).
\end{itemize}
In all cases the rate of convergence is 4, which is optimal. The error constant is slightly better if we prescribe
the boundary condition on the cut boundary, which is to be expected since this does not introduce any approximations
of the boundary condition. The difference between the other three methods is negligible.

\section*{Appendix: Verification of Some Estimates}

\paragraph{Estimate (\ref{eq:technicalA}).} 

For each $x \in \partial \Omega_h$ we have the representation 
\begin{equation}\label{appendix:a}
\phi(x)= \phi(p(x)) 
+ \int_0^1 \nabla \phi (sx+ (1-s) p(x)) \cdot (x - p(x)) \, ds
\end{equation}
Using the Cauchy-Schwarz inequality we obtain
\begin{align}
|\phi(x)|^2 &\lesssim |\phi(p(x))|^2 + \delta_h \| n \cdot \nabla \phi \|_{I_x}^2
\\ \label{appendix:b}
&\lesssim |\phi(p(x))|^2 + \delta_h \| n \cdot \nabla \phi \|_{I_{\delta_h}(p(x))}^2
\end{align}
where $I_x$ is the line segment between $x$ and $p(x)$ and 
$I_{\delta_h}(p(x))$ is the line segment between the points
$p(x)\pm \delta_h n(p(x))$. 

Integrating (\ref{appendix:b}) over $\partial \Omega_h$ we obtain 
\begin{align}
\| \phi \|^2_{\partial \Omega_h} 
&\lesssim 
\int_{\partial \Omega_h} |\phi \circ p(x)|^2 dx 
+ \int_{\partial \Omega_h} \delta_h \|n\cdot \nabla \phi \|^2_{I_{\delta_h}(p(x))} 
dx
\\
&\lesssim 
\int_{\partial \Omega} |\phi(y) |^2 dy 
+ \int_{\partial \Omega} \delta_h \|n\cdot \nabla \phi \|^2_{I_{\delta_h}(y)} dy
\\
&\lesssim 
\|\phi \|^2_{\partial \Omega} 
+ \delta_h \| n \cdot \nabla \phi \|^2_{U_{\delta_h}(\partial \Omega)}
\end{align}
where we first changed the domain of integration from $\partial \Omega_h$ 
to $\partial \Omega$ and then from the tubular coordinates to 
the Euclidian coordinates. Next integrating (\ref{appendix:b}) over $I_{\delta_h}(y)$, 
with 
$y=p(x)\in \partial \Omega $, we obtain 
\begin{equation}\label{appendix:d}
\| \phi \|^2_{I_{\delta_h(y)}} 
\lesssim \delta_h | \phi(y) |^2 
+ \delta_h^2 \| n \cdot \nabla \phi \|^2_{I_{\delta_h(y)}}
\end{equation}
Again using appropriate changes of coordinates we obtain
\begin{align}
\| \phi \|^2_{U_{\delta_h}(\partial \Omega)} 
&\lesssim 
\int_{\partial \Omega} \| \phi \|^2_{I_{\delta_h}(y)} dy
\\
&\lesssim 
 \int_{\partial \Omega} \delta_h | \phi(y) |^2 dy
+ \int_{\partial \Omega}  \delta_h^2 \| n \cdot \nabla \phi \|^2_{I_{\delta_h(y)}} dy
\\
&\lesssim \label{appendix:e}
\delta_h \| \phi \|^2_{\partial \Omega} 
+ 
\delta_h^2 \|n \cdot \nabla \phi \|^2_{U_{\delta_h}(\partial \Omega)}
\end{align}

\paragraph{Acknowledgement.} This research was supported in part by EPSRC, UK, Grant No. EP/J002313/1, the Swedish Foundation for Strategic Research Grant No.\ AM13-0029, the Swedish Research Council Grants Nos.\ 2011-4992, 
2013-4708, and Swedish strategic research programme eSSENCE.

\bibliographystyle{plain}
\bibliography{BDT}

\newpage
\begin{figure}[ht]
\begin{center}
\includegraphics[scale=0.5]{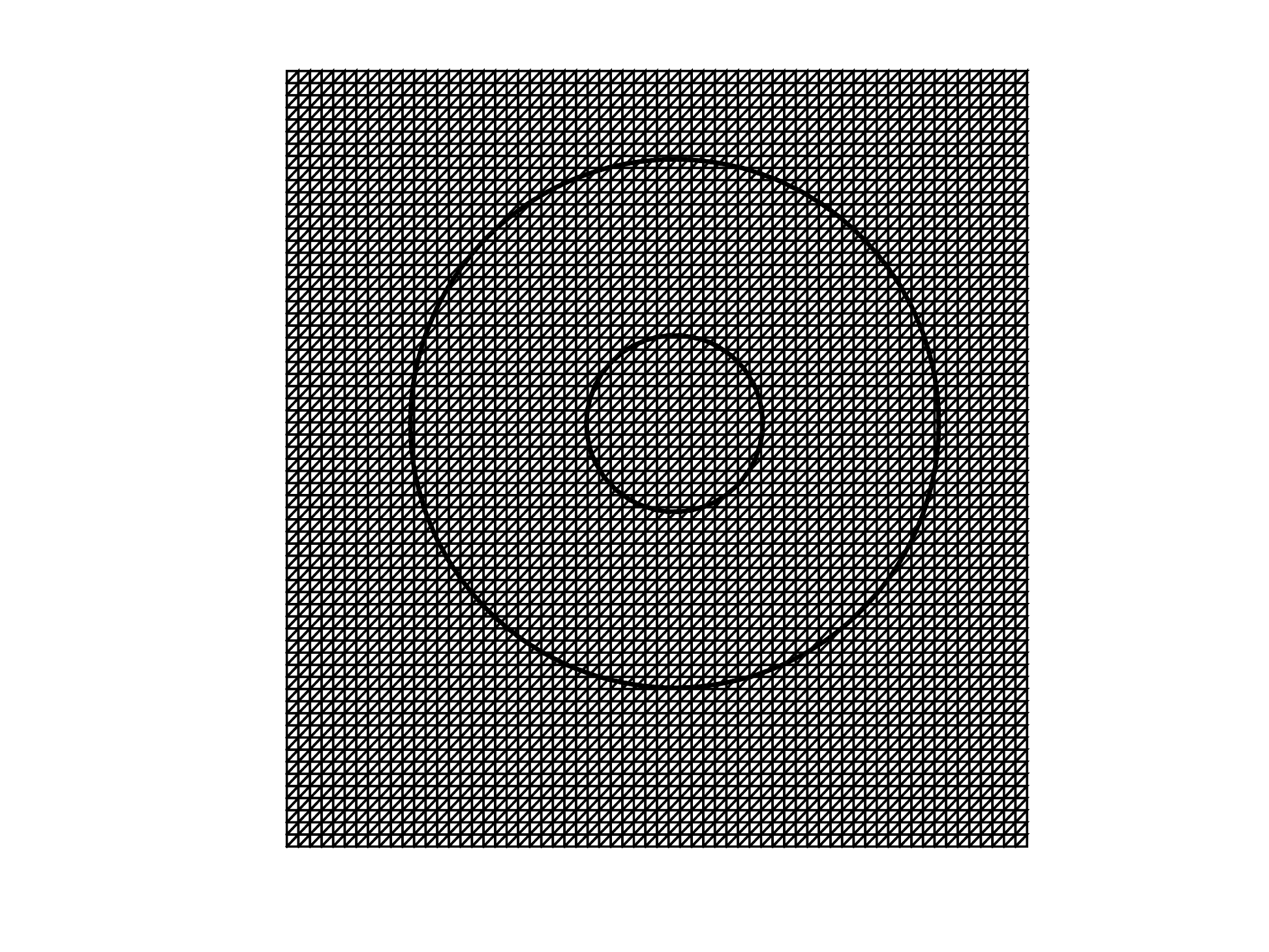}\end{center}
\begin{center}\includegraphics[scale=0.4]{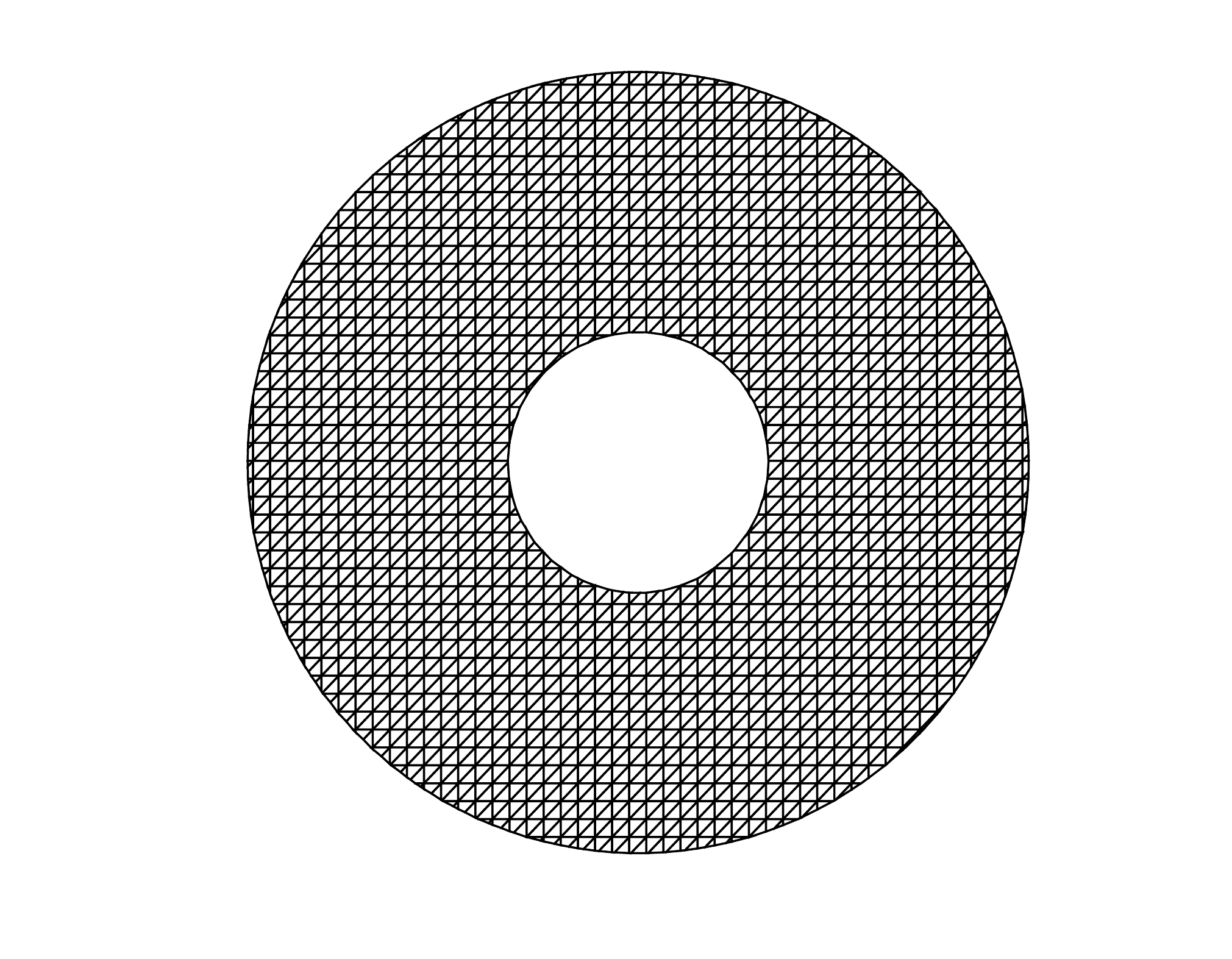}
\end{center}
\caption{Background mesh with the boundary of $\Omega$ indicated, and the corresponding computational mesh.} \label{fig:domain1}
\end{figure}
\begin{figure}[ht]
\begin{center}
\includegraphics[scale=0.25]{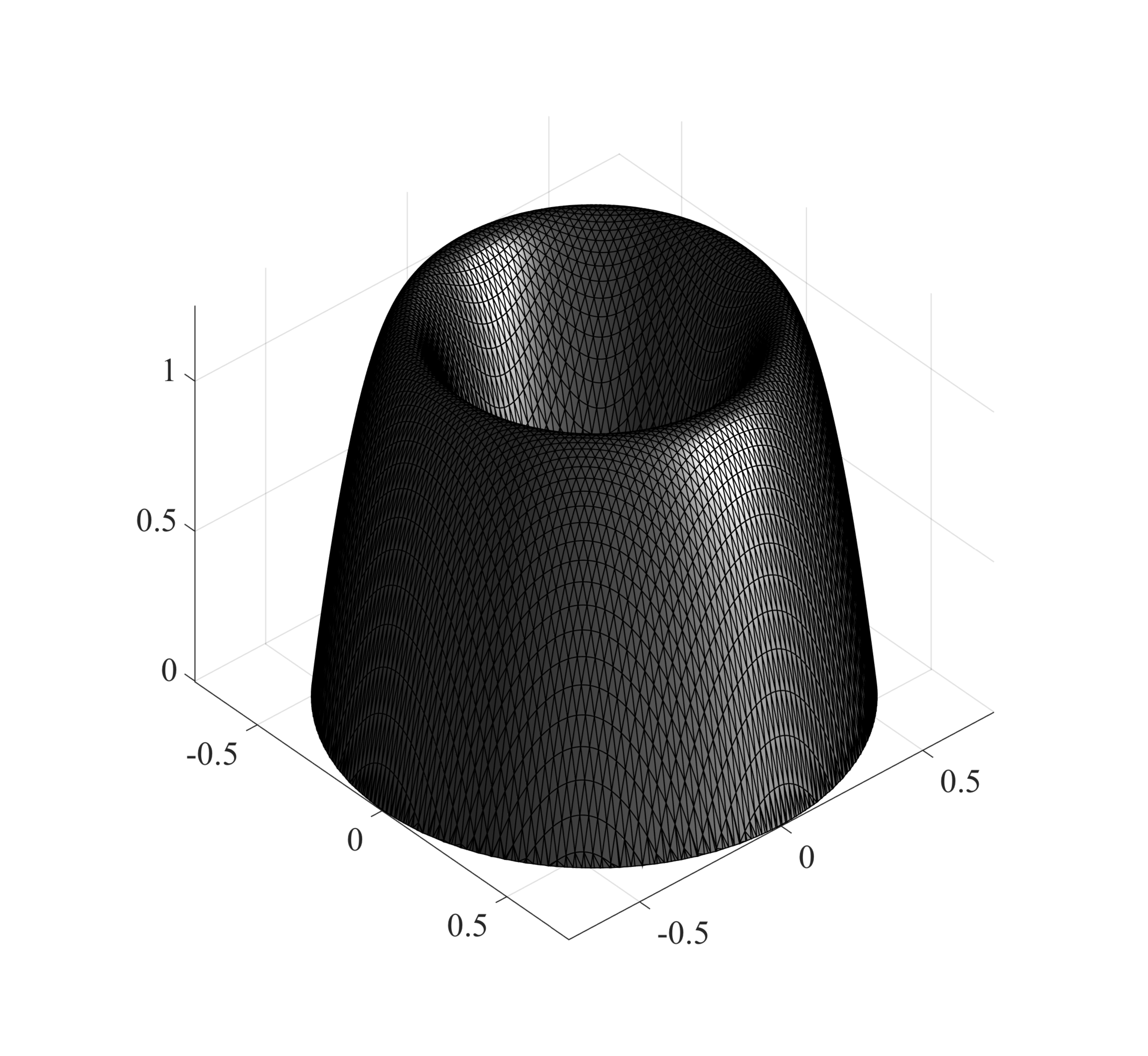}
\end{center}
\caption{Elevation of the approximate solution on one of the meshes in a sequence.} \label{fig:elev1}
\end{figure}
\begin{figure}[ht]
\begin{center}
\includegraphics[scale=0.25]{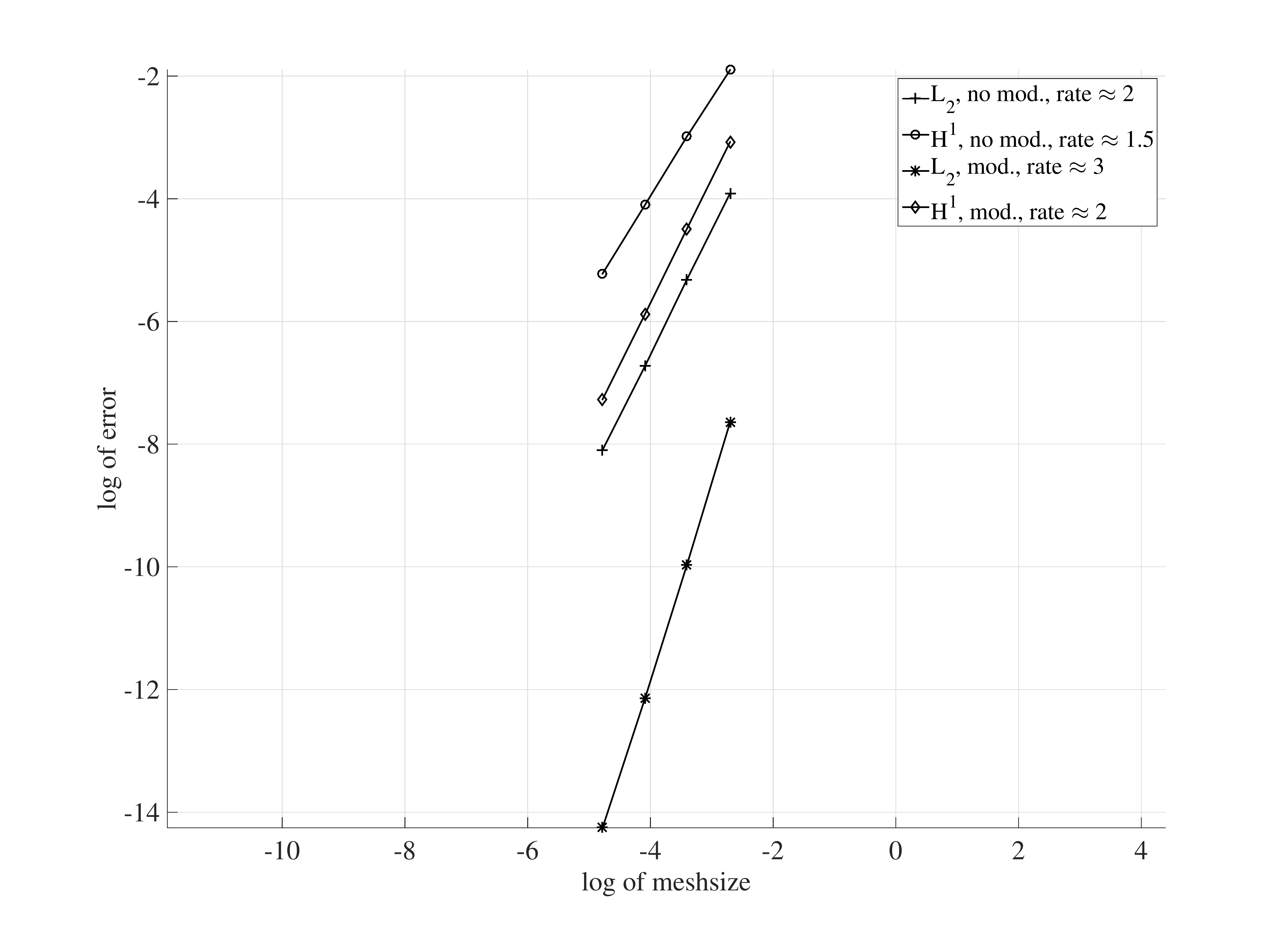}
\end{center}
\caption{Convergence using $P^2$ elements, symmetric form.} \label{fig:conv1p2}
\end{figure}
\begin{figure}[ht]
\begin{center}
\includegraphics[scale=0.25]{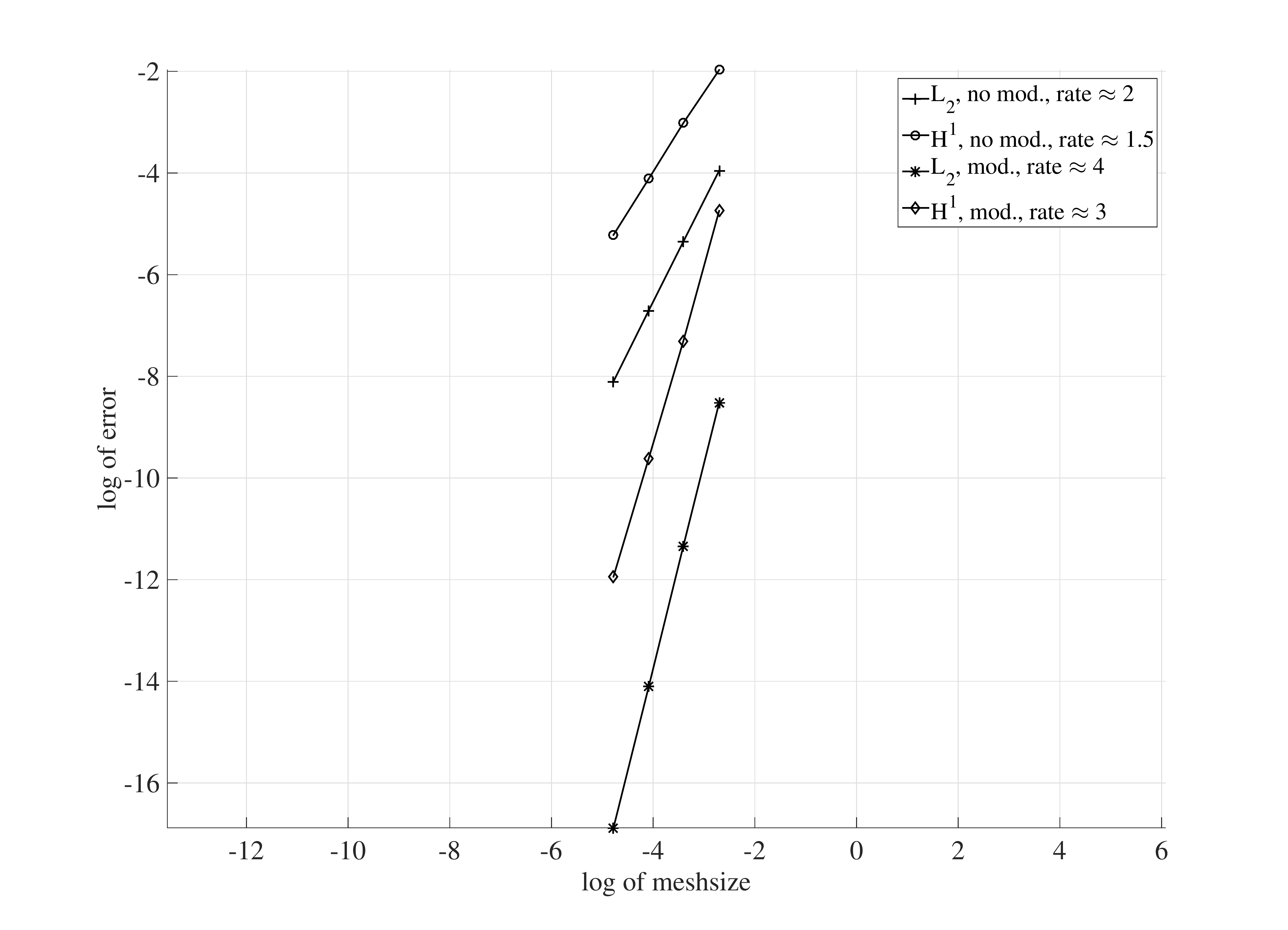}
\end{center}
\caption{Convergence using $P^3$ elements, symmetric form.} \label{fig:conv1p3}
\end{figure}

\begin{figure}[ht]
\begin{center}
\includegraphics[scale=0.5]{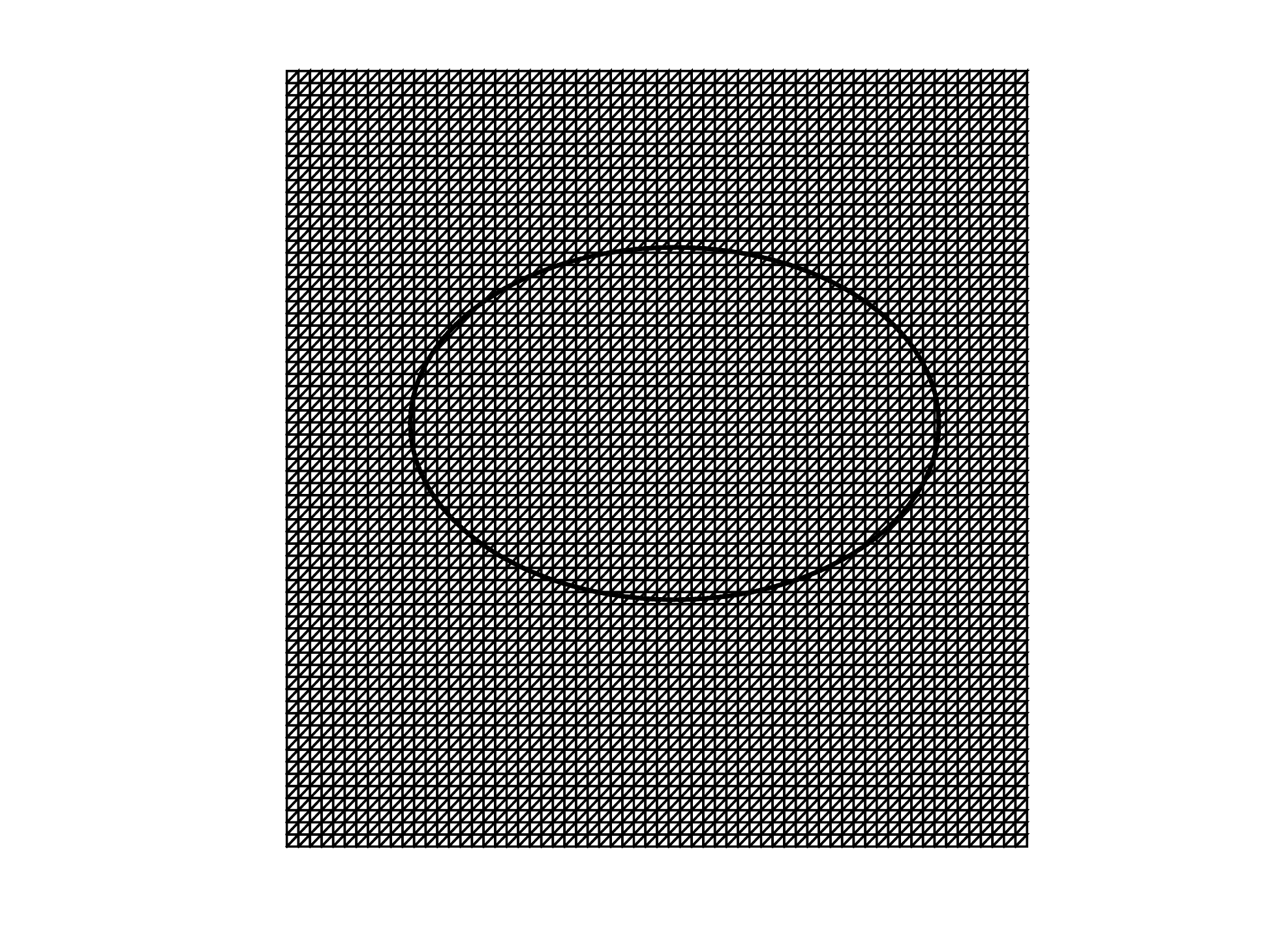}\end{center}
\begin{center}\includegraphics[scale=0.4]{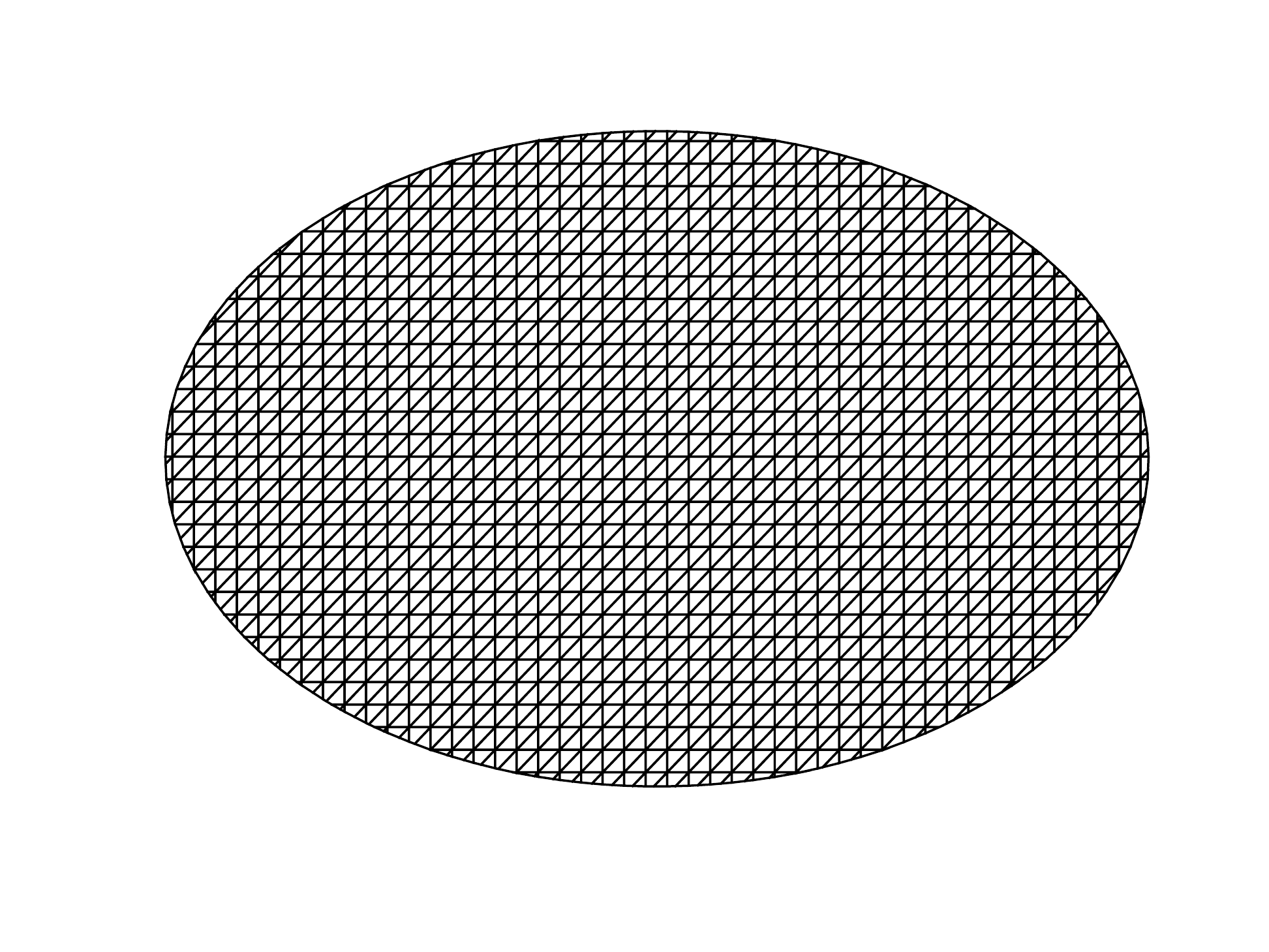}
\end{center}
\caption{Background mesh with the boundary of $\Omega$ indicated, and the corresponding computational mesh.} \label{fig:domain2}
\end{figure}

\begin{figure}[ht]
\begin{center}
\includegraphics[scale=0.25]{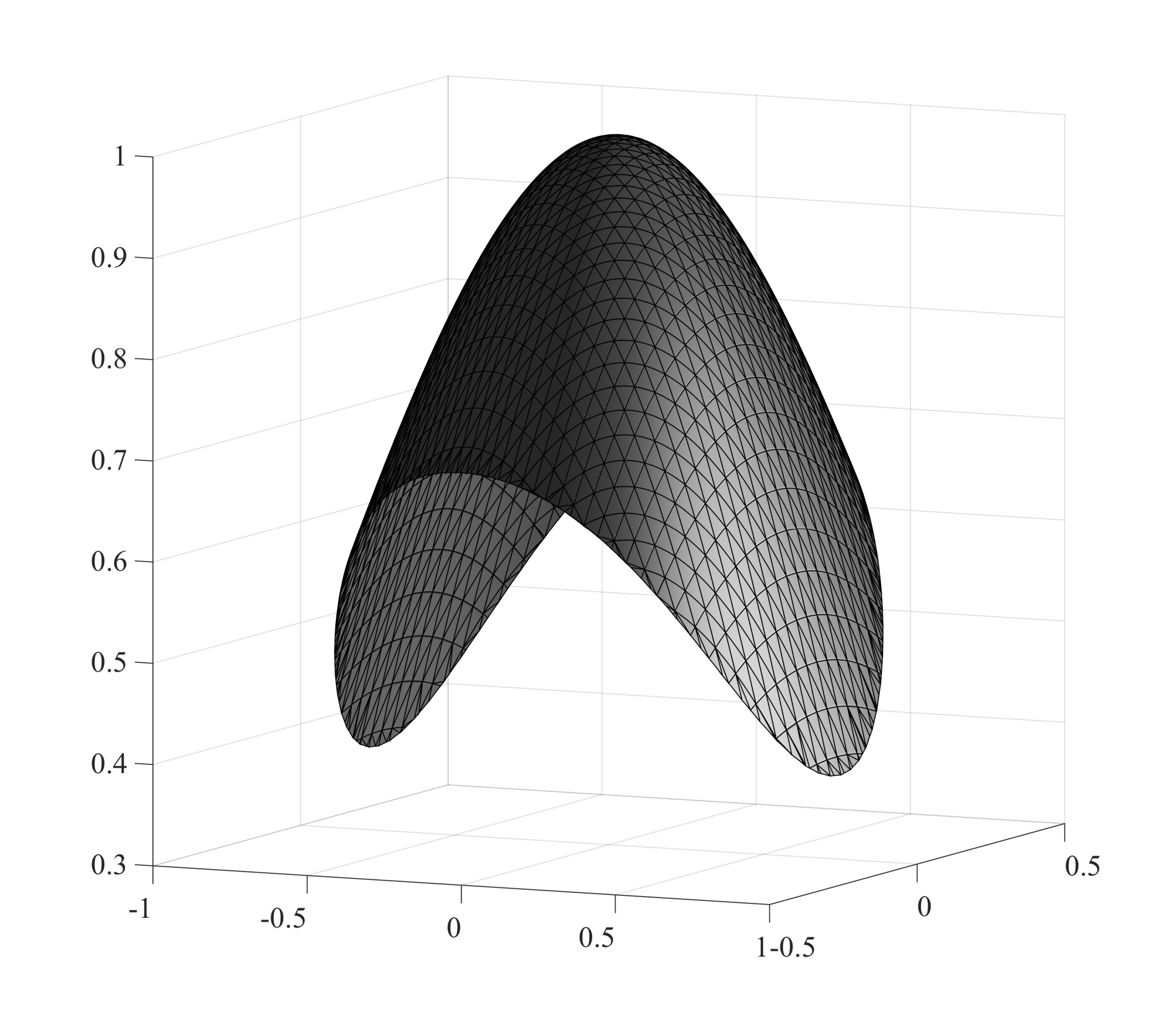}
\end{center}
\caption{Elevation of the approximate solution on one of the meshes in a sequence.} \label{fig:elev2}
\end{figure}

\begin{figure}[ht]
\begin{center}
\includegraphics[scale=0.25]{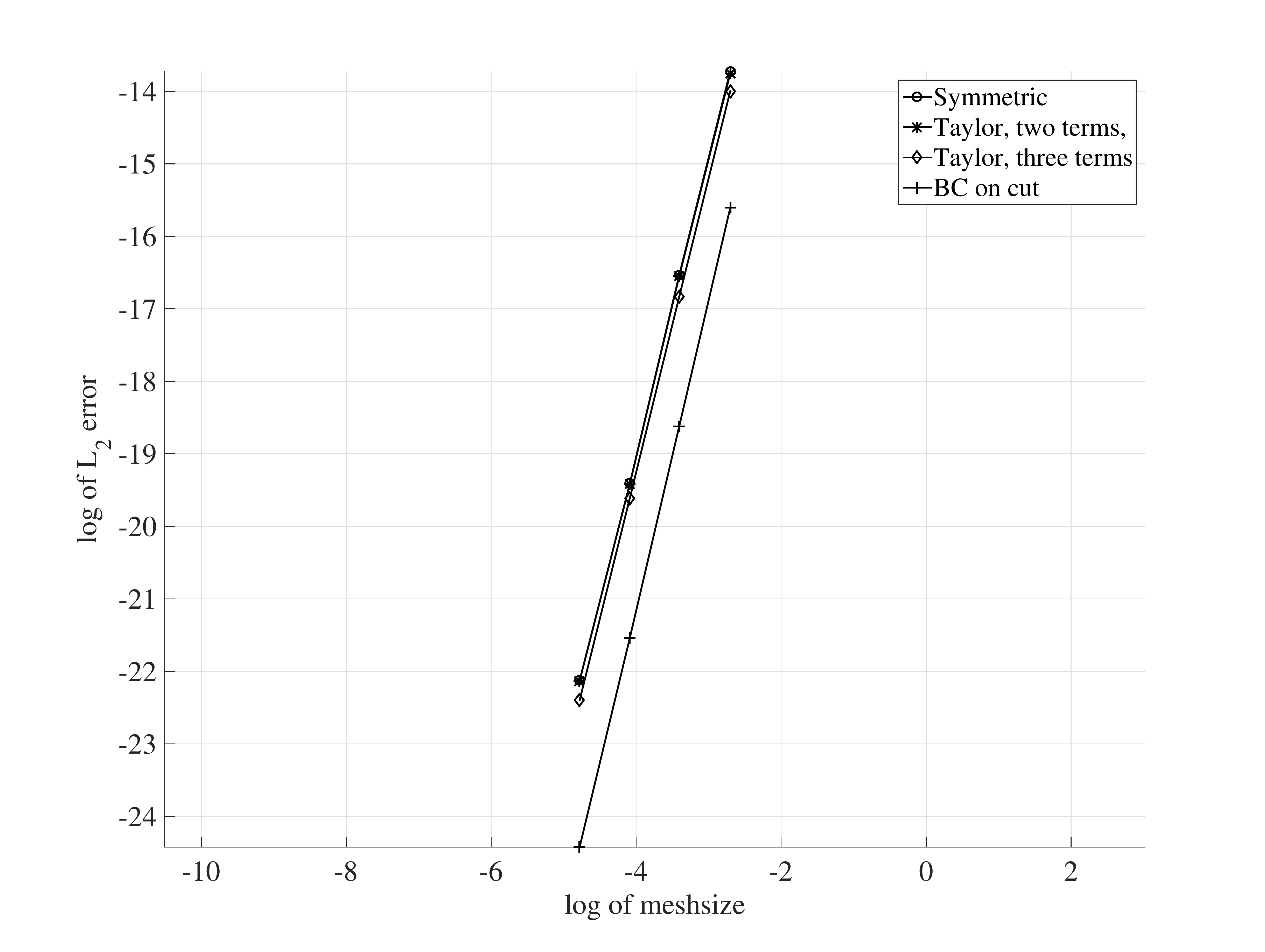}
\end{center}
\caption{Convergence in $L_2$ for four different approaches.} \label{fig:convell}
\end{figure}

\end{document}